\documentclass[a4paper,12pt]{amsart}
\usepackage[english]{babel}
\usepackage[T1]{fontenc}

\usepackage{array}
\usepackage{makecell}

\usepackage[a4paper,margin=2.0cm]{geometry}
\usepackage[justification=centering]{caption}

\usepackage{enumerate}
\usepackage{tabularx} 
\usepackage{amsmath} 
\usepackage{graphicx}
\usepackage{amsmath,amscd,amsbsy,amssymb,latexsym,url,bm,amsthm,mathrsfs}
\usepackage{epsfig,graphicx,subfigure}
\usepackage{enumerate,balance,mathtools}
\usepackage{wrapfig}
\usepackage{mathrsfs,euscript}
\usepackage[vlined,boxed,commentsnumbered,linesnumbered,ruled]{algorithm2e}
\usepackage{amsfonts}
\usepackage{amsthm}
\usepackage{fancyhdr}
\usepackage{picinpar}
\usepackage{listings}
\usepackage{layout}
\usepackage[all]{xy}
\usepackage{indentfirst}
\usepackage{tikz-cd}
\usepackage{tikz}
\usepackage[all]{xy}
\usepackage{mathrsfs}
\usepackage{hyperref}
\hypersetup{
colorlinks=true,
linkcolor=[RGB]{0,0,204},
filecolor=magenta,      
urlcolor=cyan,
citecolor=[RGB]{0,204,0},
}
\usepackage{setspace}

\usepackage{amsmath,amsthm,amssymb,amsxtra,calligra,mathrsfs}
\usepackage{graphicx}
\usepackage{tikz}
\usepackage{tikz-cd} 
\usepackage{xcolor}
\usepackage{upgreek}
\usepackage{comment}
\usepackage{graphicx}
\usepackage{mathtools}
\usepackage{extarrows}
\usepackage{faktor}
\usepackage{mathrsfs,euscript}
\usepackage{comment}
\usepackage{bookmark}
\usepackage{mathrsfs}
\usepackage{multirow}

\newcommand{\wt}[1]{\widetilde{#1}}

\newcommand{\mb}[1]{\mathbb{#1}}

\newcommand{\ove}[1]{\overline{#1}}

\newcommand{\mtc}[1]{\mathcal{#1}}
\newcommand{\mtf}[1]{\mathfrak{#1}}
\newcommand{\mts}[1]{\mathscr{#1}}

\DeclareMathOperator{\coeff}{coeff}
\DeclareMathOperator{\Fut}{Fut}

\DeclareMathOperator{\PGL}{PGL}

\DeclareMathOperator{\Spec}{Spec}
\DeclareMathOperator{\Proj}{Proj}

\DeclareMathOperator{\Pic}{Pic}
\DeclareMathOperator{\Hom}{Hom}

\DeclareMathOperator{\Ext}{Ext}

\DeclareMathOperator{\bfV}{{\textbf{V}}}

\DeclareMathOperator{\Bs}{Bs}

\DeclareMathOperator{\sst}{ss}
\DeclareMathOperator{\Id}{Id}
\DeclareMathOperator{\NS}{NS}

\DeclareMathOperator{\GL}{GL}
\DeclareMathOperator{\SL}{SL}
\DeclareMathOperator{\Sym}{Sym}

\DeclareMathOperator{\GIT}{GIT}

\DeclareMathOperator{\Aut}{Aut}
\DeclareMathOperator{\Jac}{Jac}

\DeclareMathOperator{\vol}{vol}
\DeclareMathOperator{\Gr}{Gr}

\DeclareMathOperator{\CM}{CM}
\DeclareMathOperator{\Bl}{Bl}

\newcommand{\bG}{\mathbb{G}}

\newcommand{\sslash}{\mathbin{\mkern-3mu/\mkern-6mu/\mkern-3mu}}

\newcommand{\sheafHom}{\mathscr{H}\text{\kern -3pt {\calligra\large om}}\,}

\newcommand{\chapterauthor}[1]{%
  {
    \parindent0pt\vspace*{-10pt}%
    \linespread{1}\scshape by #1%
  }
}

\DeclareMathOperator{\Hilb}{\mathrm{Hilb}}

\newcommand{\bV}{{\mathbb V}}

\newcommand{\bA}{{\mathbb A}}
\newcommand{\bZ}{{\mathbb Z}}
\newcommand{\bN}{{\mathbb N}}
\newcommand{\bC}{{\mathbb C}}
\newcommand{\bQ}{{\mathbb Q}}
\newcommand{\bP}{{\mathbb P}}
\newcommand{\bR}{{\mathbb R}}

\newcommand{\sL}{{\mathscr{L}}}

\newcommand{\sX}{{\mathscr{X}}}

\newcommand{\sV}{{\mathscr{V}}}

\newcommand\ot{{\otimes}}

\usepackage{dutchcal}

\newtheorem{theorem}{Theorem}[section]
\newtheorem{theodef}[theorem]{Theorem-Definition}
\newtheorem{lemma}[theorem]{Lemma}

\newtheorem{corollary}[theorem]{Corollary}
\newtheorem{prop}[theorem]{Proposition}

\newtheorem{question}[theorem]{Question}
\newtheorem{example}[theorem]{Example}

\newtheorem{remark}[theorem]{Remark}

\newtheorem{Prbm}{\bf Problem}

\theoremstyle{definition}
\newtheorem{defn}[theorem]{Definition}

\theoremstyle{remark}

\title{K-stability of Thaddeus' moduli of stable bundle pairs on genus two curves}
\date{\today}

\author{Junyan Zhao \\ with an Appendix by Benjamin Church and Junyan Zhao}

\begin{document}

\begin{abstract}

The moduli space of bundle stable pairs $\overline{M}_C(2,\Lambda)$ on a smooth projective curve $C$, introduced by Thaddeus, is a smooth Fano variety of Picard rank two. Focusing on the genus two case, we show that its K-moduli space is isomorphic to a GIT moduli of lines in quartic del Pezzo threefolds. Additionally, we construct a natural forgetful morphism from the K-moduli of $\overline{M}_C(2,\Lambda)$ to that of the moduli spaces of stable vector bundles $\overline{N}_C(2,\Lambda)$. In particular, Thaddeus' moduli spaces for genus two curves are all K-stable.

\end{abstract}

\maketitle

\tableofcontents

\section{Introduction}

K-stability, introduced by differential geometers, is an algebro-geometric condition that characterizes the existence of Kähler-Einstein metrics on Fano varieties (ref. \cite{Tia97, Don02}). In the past decade, the algebraic theory of K-stability has seen significant progress, leading to the construction of a projective moduli space — commonly referred to as the \emph{K-moduli space} — which parametrizes K-polystable Fano varieties of a fixed dimension and volume. For an overview of these advancements, see \cite{Xu21, Xu25}.

Despite the completion of the general theory, to prove the K-stability of Fano varieties, even Fano hypersurfaces, is still a technically hard problem. A powerful approach for verifying explicit K-stability using admissible flags was developed in \cite{AZ22}; for a more detailed exposition and applications to Fano threefolds, see \cite{CA21}. An aligning problem is to explicitly describe K-moduli spaces for smooth Fano varieties, in other words, find all the K-(semi/poly)stable degenerations. One promising approach to this is the moduli continuity method, which proceeds as follows. 
\begin{enumerate}
    \item First, using equivariant K-stability (ref. \cite{Zhu21}), one can typically identify a K-stable member within a given family of Fano varieties. By the openness of K-(semi)stability (ref. \cite{BL18, BLX19, Xu20}), this implies that general members of the family are also K-stable.
    \item Understand the geometry of the K-semistable limits, especially make sure that all the limits appear in some appropriate parameter space. This requires the singularity estimate introduced in \cite{Liu18}, built on earlier work in \cite{Fuj18}. 
    \item Next, by using the ampleness of the CM line bundle (ref. \cite{CP21, XZ20}), one can propose a candidate for the K-moduli space based on GIT stability with respect to the CM line bundle over the parameter space.
    \item Finally, establish an isomorphism between the K-moduli space and the GIT moduli space, using the separatedness and properness of K-moduli spaces (ref. \cite{BX19, ABHLX20, LXZ22}). It is worth mentioning that at this step, there is a significant technical issue, first pointed out in \cite[Appendix A]{SS17}.
    \item[(4+)] Additionally, one can further study the GIT stability of the objects and the corresponding GIT moduli space. Typically, smooth objects are GIT stable, which allows us to prove the K-stability of all smooth Fano varieties in that family.
\end{enumerate}  

For any dimension $N=3g-3$ with $g\geq 2$, there are two special families of smooth Fano varieties coming from vector bundles: the moduli spaces of rank two stable vector bundles with a fixed odd determinant, and Thaddeus' moduli spaces of \emph{bundle stable pairs}\footnote{In Thaddeus' original paper \cite{Tha94}, this is called \emph{stable pairs}. Here we use \emph{bundle stable pairs} to avoid confusion with the stable pairs in the sense of birational geometry.} (see Definition \ref{defn:thaddeus}) over genus $g$ curves. The former has Picard rank one and Fano index two, while the latter is obtained by blowing up the former along the Brill--Noether locus. A significant area of interest is the study of the K-stability of these Fano varieties (see \cite[Problem 11.2]{AIM}), and conjecturally those two families of smooth moduli spaces are all K-stable for each $g\geq 2$.

However, proving this in general does not seem approachable. For $g=2$, the moduli space 
$\ove{N}_C(2,\Lambda)$ of rank two stable vector bundles with a fixed odd determinant $\Lambda$ over a smooth genus $2$ curve $C$ is isomorphic to a $(2,2)$-complete intersection in $\bP^5$ (or equivalently a quartic del Pezzo threefold), which belongs to family №1.14 in the classification of smooth Fano threefolds. Furthermore, the K-moduli space of quartic del Pezzo varieties of any dimension is isomorphic to the GIT moduli space; see \cite[Corollary 4.1 \& Proposition 4.2]{SS17}. In particular, the moduli space $\ove{N}_C(2,\Lambda)$ for any smooth genus $2$ curve $C$ is K-stable.

In this paper, we study the K-stability and the K-moduli stack (and space) of Thaddeus' moduli of bundle stable pairs over genus two curves, which are smooth Fano threefolds of Family №2.19 in the   Iskovskikh--Mori--Mukai classification. Geometrically, these varieties are blow-ups of $\bP^3$ along genus $2$ curves of degree $5$ and admit a Sarkisov link structure that identifies them as blow-ups of $(2,2)$-complete intersections in $\bP^5$ along a line. Our first main result establishes that every member in this K-moduli space is also isomorphic to the blow-up of $\bP^3$ and has the same Sarkisov link structure as the smooth ones.

\begin{theorem}\label{thm:main1}
 Let $\mts{M}^K_{\textup{№2.19}}$ be the K-moduli stack of the family №2.19, and $\ove{M}^K_{\textup{№2.19}}$ be its good moduli space. Then the followings hold.
\begin{enumerate}
    \item Every K-semistable Fano variety $X$ in $\mts{M}^K_{\textup{№2.19}}$ is isomorphic to $\Bl_{C}\mb{P}^3$ for some space curve $C$, whose syzygy is of the form $$0\ \longrightarrow \ \mtc{O}_{\bP^3}(-4)^{\oplus2}\ \longrightarrow \ \mtc{O}_{\bP^3}(-3)^{\oplus2}\oplus \mtc{O}_{\bP^3}(-2)\ \longrightarrow \ \mtc{I}_{C} \ \longrightarrow \ 0,$$ where the unique quadric surface containing $C$ is normal. Moreover, $X$ can be obtained by blowing up a $(2,2)$-complete intersection in $\bP^5$ along a line.
    \item The K-moduli stack $\mts{M}^K_{\textup{№2.19}}$ is a smooth connected component of $\mts{M}^K_{3,26}$, and the K-moduli space $\ove{M}^K_{\textup{№2.19}}$ is a normal projective variety.
\end{enumerate}
\end{theorem}

Based on Theorem \ref{thm:main1}, it is natural to consider the parameter space $W$, which consists of pairs $(V,\ell)$, where $V$ is a $(2,2)$-complete intersection containing a line $\ell$. This space forms a large open subset of a Grassmannian bundle over a Grassmannian (see Section \S \ref{Sec:5} for a formal set-up). Furthermore, the associated GIT quotients $\ove{M}^{\GIT}(t):= W\sslash_t\PGL(6)$, which depend on a parameter $t$ denoting a polarization $\eta+t\cdot\xi$, yield natural projective birational models for the K-moduli spaces. Our second result identifies the K-moduli space (resp. stack) with one of these (resp. stacky) GIT quotients and describes the K-(semi/poly)stable elements of this family, including their $\bQ$-Gorenstein degenerations.

\begin{theorem}\label{thm:main2}
  Let $V$ be a $(2,2)$-complete intersection in $\mb{P}^5$, $\ell\subseteq V$ be a line, and $X:=\Bl_{\ell}V$ be the blow-up of $V$ along $\ell$. Then $X$ is K-(semi/poly)stable if and only if $(V,\ell)$ is an $\epsilon$-GIT (semi/poly)stable point in $W$, where $\epsilon>0$ is any sufficiently small rational number. Moreover, there is a natural isomorphism $$\mts{M}^K_{\textup{№2.19}}\ \simeq\   \mts{M}^{\GIT}(\epsilon),$$ which descends to an isomorphism between their good moduli spaces $$\ove{M}^K_{\textup{№2.19}}\ \simeq\   \ove{M}^{\GIT}(\epsilon).$$
  In particular, the followings hold.
  \begin{enumerate}
      \item Every smooth member in this deformation family of Fano threefolds is K-stable.
      \item If $X=\Bl_{\ell}V$ is K-semistable, then the quartic del Pezzo threefold $V$ is K-semistable, and $\ell$ is contained in the smooth locus of $V$.
      \item There is a commutative diagram
      $$\xymatrix{\mts{M}^K_{\textup{№2.19}} \ar[rr] \ar[d] && \mts{M}^K_{\textup{№1.14}} \ar[d] \\
\ove{M}^K_{\textup{№2.19}}\ar[rr]  && \ove{M}^K_{\textup{№1.14}}
      }$$ composed by the forgetful map $[(V,\ell)]\mapsto [V]$, good moduli space morphisms, and the morphism on good moduli spaces induced by universality.
  \end{enumerate}

\end{theorem}

As an immediate consequence of the above results, we are able to answer the question about K-stability of Thaddeus' moduli of bundle stable pairs over genus $2$ curves.

\begin{corollary}
    For any smooth curve $C$ of genus $2$ and any fixed odd determinant $\Lambda\in \Pic(C)$, the Thaddeus' moduli space of bundle stable pair $\ove{N}_C(2,\Lambda)$ is a K-stable Fano variety. In particular, $\ove{M}^K_{\textup{№2.19}}$ yields a modular compactification of the space parametrizing Thaddeus' moduli spaces of bundle stable pairs over smooth curves of genus $2$.
\end{corollary}

\noindent We remark that the main focus of this paper is the isomorphism between the K-moduli space and a certain VGIT quotient. We do not explicitly describe the singular K-polystable degenerations, which can likely be fully classified by analyzing the GIT polystable objects; see Section~\S7.1.

\subsection*{Non-isotrivial families of smooth K-stable varieties}

In the appendix, we discuss a question of Javanpeykar and Debarre in the context of Fano family \textnumero2.19.

\begin{theorem} The following holds:
\begin{enumerate}
    \item there is a non-isotrivial family of smooth Fano threefolds №2.19 parametrized by a stacky $\bP^1$, i.e.\ a proper DM stack $\mtc{P}^1$ with finitely many stacky points, whose coarse space is isomorphic to $\bP^1$,
    \item there is a non-constant morphism from $\bP^1 \to \ove{M}^{K}_{\textup{№2.19}}$ whose image is contained in the locus parametrizing smooth K-stable Fano threefolds of family №2.19. In fact, such curves are dense and occupy infinitely many numerical classes. However, none lift to $\mts{M}^K_{\textup{№2.19}}$ and moreover,
    \item there are no non-trivial families of smooth Fano threefolds in №2.19 over $\bP^1$.
\end{enumerate}
\end{theorem}

\subsection*{Prior and related works}

The moduli continuity method has been successfully applied to several families, including the case of dimension two \cite{MM93, OSS16}, as well as specific families in higher dimensions \cite{LX19,SS17,ADL22,Liu22,LZ24}. In particular, in \cite{Liu22,LZ24}, the geometry and moduli space of K3 surfaces are widely used, since they appear as sections in the anticanonical linear series of Fano threefolds.

Recent work on K-moduli spaces for Fano threefolds has also employed a different approach known as the Abban–Zhuang method \cite{AZ22}. This method relies on stability threshold estimates using inversion of adjunction and multi-graded linear systems. Examples include \cite{Pap22, ACD+23, CT23,ACKLP, CDG+, CFFK, DJKQ}.

We also remark here that the K-stability of the family \textnumero2.19 is studied via Abban--Zhuang's method in \cite{GDGV24}.

\subsection*{Outline of the paper}

We begin with a review of concepts and facts in Section \S 2, including K3 surfaces, Fano threefolds, and the notion of K-stability and K-moduli. In \S 3, we discuss the moduli spaces of polarized K3 surfaces, which naturally arise in the study of Fano threefolds, particularly in their anti-canonical linear systems. Section \S 4 focuses on the K-semistable limits of a family of Fano threefolds, proving that every limit can be obtained by certain blow-ups as in the smooth case. In \S 5, we study the variation of GIT stability for lines in quartic del Pezzo threefolds and the connections between GIT and K-moduli spaces. In \S 6, we prove that the K-moduli stack for the Fano threefold family №2.19 is smooth and isomorphic to a GIT moduli stack. Finally, in Section \S 7, we discuss possible extensions of our results and pose open questions for future research.

\subsection*{Acknowledgements}

The author would like to express gratitude to Yuchen Liu for the assistance in proving some technical lemmas, and to Jenia Tevelev for asking the author this question about K-stability of moduli of bundle stable pairs during the AGNES workshop. Additionally, the author thanks Theodoros Papazachariou for fruitful discussions.

\section{Preliminaries}

We work over the field $\bC$ of complex numbers. A projective space $\bP^n$ parametrizes $1$-dimensional quotient spaces of a vector space $\bC^{n+1}$.

\subsection{K3 surfaces} We begin by reviewing some basics in the geometry of (quasi-)polarized K3 surfaces, which arise naturally in the study of Fano threefolds.

\begin{defn}
  A \emph{K3 surface} is a normal projective surface $X$ with at worst ADE singularities satisfying $\omega_X\simeq \mathcal{O}_X$ and $H^1(X, \mathcal{O}_X) =0$. A polarization (resp. quasi-polarization) on a K3 surface $X$ is an ample (resp. big and nef) line bundle $L$ on $X$. We call the pair $(X,L)$ a \emph{polarized} (resp. \emph{quasi-polarized}) \emph{K3 surface of degree $d$}, where $d=(L^2)$. Since $d$ is always an even integer, we sometimes write $d = 2k$.
\end{defn}

\begin{theorem}[cf. \cite{May72, SD}]\label{Mayer}
Let $(S,L)$ be a polarized K3 surface of degree $2k$. Then one of the followings holds.
\begin{enumerate}
    \item \textup{(Generic case)} The linear series $|L|$ is very ample, and the embedding $\phi_{|L|}:S\hookrightarrow |L|^{\vee}$ realizes $S$ as a degree $2k$ surface in $\mb{P}^{k+1}$. In this case, a general member of $|L|$ is a smooth non-hyperelliptic curve.
    \item \textup{(Hyperelliptic case)} The linear series $|L|$ is base-point-free, and the induced morphism $\phi_{|L|}$ realizes $X$ as a double cover of a normal surface of degree $k$ in $\mb{P}^{k+1}$. In this case, a general member of $|L|$ is a smooth hyperelliptic curve, and $|2L|$ is very ample.
    \item \textup{(Unigonal case)}  The linear series $|L|$ has a base component $E$, which is a smooth rational curve. The linear series $|L-E|$ defines a morphism $S\rightarrow \mb{P}^{k+1}$ whose image is a rational normal curve in $\mb{P}^{k+1}$. In this case, a general member of $|L-E|$ is a union of disjoint elliptic curves, and $|2L|$ is base-point-free.
\end{enumerate}
\end{theorem}

\subsection{Geometry of Fano threefolds}\label{sec:2.19}

In this section, let us briefly review the geometry of a smooth member in the deformation family №2.19 of Fano threefolds (ref. \cite[Section \S4.4]{CA21}).

\begin{lemma}
    Any smooth curve $C$ in $\mb{P}^3$ of genus two and degree five is contained in a unique normal quadric surface $Q$.
    \begin{enumerate}
        \item If $Q$ is isomorphic to $\mb{P}^1\times\mb{P}^1$, then $C$ is of type $(2,3)$ as a divisor in $\mb{P}^1\times\mb{P}^1$;
        \item If $Q$ is isomorphic to the quadric cone $\mb{P}(1,1,2)$, then $C$ is of type $\mtc{O}_{\mb{P}(1,1,2)}(5)$ as a divisor in $\mb{P}(1,1,2)$, which contains the vertex of $\mb{P}(1,1,2)$ as a smooth point. 
    \end{enumerate} 
\end{lemma}

Let $C\subseteq \mb{P}^3$ be the smooth genus $2$ degree $5$ curve on a normal quadric surface $Q$. Consider the linear series $|\mtc{I}_{C}(4)|$, a general member $W$ in which is a quartic K3 surface containing $C$. Then the residue curve of $C$ with respect to $W$ and the quadric $Q$ is a rational normal curve $\Gamma$. Let $\pi: X\rightarrow \mb{P}^3$ the blow-up of $\mb{P}^3$ along $C$ with exceptional divisor $E$, and $\wt{Q}$ be the strict transform of $Q$. The linear system $|\mtc{I}_{C}(3)|$ of cubic surfaces vanishing along C induces a rational map $\mb{P}^3\dashrightarrow \mb{P}^5$, which is resolved by the blow-up $\pi$: $$\xymatrix{
& & X \ar[dl]_{\pi} \ar[dr]^{\phi} &\\
 & \mb{P}^3 \ar@{-->}[rr]^{|I_C(3)|}  &  &  V\subseteq \mb{P}^5.
 }$$ The morphism $\phi$ contracts $\wt{Q}$, whose image is a line $\ell$ contained in $V$, and is an isomorphism elsewhere, and $V$ is a smooth complete intersection of two quadric hypersurfaces. Let $H=\pi^{*}\mtc{O}_{\mb{P}^3}(1)$ and $L=\phi^{*}(\mtc{O}_{\mb{P}^5}(1)|_V)$ be two big and nef divisors. Then we have the linear equivalence relations $$-K_X\sim 4H-E,\quad L\sim 3H-E,\quad \wt{Q}\sim 2H-E.$$
 Let $S\in|-K_X|$ be a smooth member. Then $\wt{C}:=S\cap E$ is mapped by $\pi$ to $C$ isomorphically, and $\Gamma:=S\cap \wt{Q}$ is mapped by $\phi$ isomorphically to the line $\ell$ of $\ove{S}:=\phi(S)$. We have that $(S,-K_X|_S)$ is a polarized K3 surface of degree $26$, and $(S,L|_S)$ is a polarized K3 surface of degree $8$.

 Conversely, Let $V \subseteq \mb{P}^5$ be a smooth complete intersection of two quadrics containing a line $\ell$. Projecting from $\ell$ induces a birational morphism
 $$\pi_{\ell}:\Bl_{\ell}V\ \longrightarrow\ \mb{P}^3,$$ which blows down all the lines in $V$ incident to $\ell$. Fix coordinates $[x_0,...,x_5]$ of $\mb{P}^5$ such that $\ell=\mb{V}(x_2,...,x_5)$ and $$V = \mb{V}\big(L_{0}x_0+L_{1}x_1 + Q, L'_{0}x_0 + L'_{1}x_1 + Q' \big),$$ where the $L_{i}$ and $L'_i$ are linear forms and the $Q,Q'$ are quadratic forms in $x_2,...,x_5$. The exceptional locus of $\pi_{\ell}$ in $\mb{P}^3$ is given by $2\times 2$ minors of the matrix 
 $$\begin{pmatrix}
   L_{0}  &  L_{1}  &  Q \\
   L'_{0}  &  L'_{1}  &  Q'
 \end{pmatrix},$$ which is a curve of genus $2$ and degree $5$.

In fact, one has the following more general correspondence.
 
\begin{lemma}\label{lem:sarkisov link}
     There is a one-to-one correspondence between the set of curves $C\subseteq \bP^3$, whose ideal is generated by the $2\times 2$ minors of the matrix 
 $$\begin{pmatrix}
   L_{0}  &  L_{1}  &  Q \\
   L'_{0}  &  L'_{1}  &  Q'
 \end{pmatrix},$$ up to projective equivalence in $\mb{P}^3$, and the set of pairs $(V,\ell)$, where $V$ is a $(2,2)$-complete intersection in $\mb{P}^5$ and $\ell$ is a line contained in $V$ such that $V$ is generically smooth along $\ell$, up to projective equivalence in $\mb{P}^5$. Moreover, this one-to-one correspondence is given by a Sarkisov link structure: the blow-up of $\mb{P}^3$ along $C$ is identified with the blow-up of $V$ along $\ell$.
\end{lemma}

\begin{proof}
    Notice that the syzygy of each such a curve $C$ is of the form $$0\ \longrightarrow \ \mtc{O}_{\bP^3}(-4)^{\oplus2}\ \longrightarrow \ \mtc{O}_{\bP^3}(-3)^{\oplus2}\oplus \mtc{O}_{\bP^3}(-2)\ \longrightarrow \ \mts{I}_{C} \ \longrightarrow \ 0.$$ Then one can argue as in \cite[Lemma 4.17]{LZ24}.
\end{proof}

\begin{lemma}\label{isomofL}
    Let $S\in|-K_X|$ be a general member. Then
    \begin{enumerate}
        \item we have a natural isomorphism $$H^0(X,L)\stackrel{\simeq}{\longrightarrow} H^0(S,L|_S)$$ of $6$-dimensional vector spaces; and
        \item the group $\NS(X)$ is free of rank $2$, generated by $[H|_S]$ and $[E|_S]$, and contains a polarization $(4H-E)|_S$ of degree $26$.
    \end{enumerate}
\end{lemma}

\begin{proof}
    Notice that there exists an exact sequence $$0\longrightarrow \mtc{O}_X(L-S)\longrightarrow \mtc{O}_X(L)\longrightarrow \mtc{O}_S(L)\longrightarrow 0,$$ where $L-S\sim -H$ is anti-effective. By Kawamata-Viehweg vanishing theorem, we have that $$h^0(X,\mtc{O}_X(-H))=h^1(X,\mtc{O}_X(-H))=0.$$ Then the desired isomprhism follows immediately. The dimension of the vector spaces is $\frac{1}{2}(L|_S)^2+2=6$. To prove the second statement, it suffices to observe that a general element $S\in|-K_X|$ is identified with a general quartic K3 surface in $\mb{P}^3$ containing a smooth rational normal curve $\Gamma$.
    
\end{proof}

\subsection{K-stability and K-moduli spaces}

\begin{defn}
    A \emph{$\mb{Q}$-Fano variety} (resp. \emph{weak $\mb{Q}$-Fano variety}) is a  normal projective variety $X$ such that the anti-canonical divisor $-K_X$ is an ample (resp. a big and nef) $\mb{Q}$-Cartier divisor, and $X$ has klt singularities.

    By \cite{BCHM}, for a weak $\bQ$-Fano variety $X$, the anti-canonical divisor $-K_X$ is always big and semiample whose ample model gives a $\bQ$-Fano variety $\overline{X} :=\Proj R(-K_X)$. We call $\overline{X}$ the \emph{anti-canonical model} of $X$.
\end{defn}

\begin{defn}
    A $\mb{Q}$-Fano variety $X$ (resp. weak $\bQ$-Fano variety) is called \emph{$\mb{Q}$-Gorenstein smoothable} if there exists a projective flat morphism $\pi:\mts{X}\rightarrow T$ over a pointed smooth curve $(0\in T)$ such that the following conditions hold:
    \begin{itemize}
        \item $-K_{\mts{X}/T}$ is $\mb{Q}$-Cartier and $\pi$-ample (resp. $\pi$-big and $\pi$-nef);
        \item $\pi$ is a smooth morphism over $T^\circ:=T\setminus \{0\}$; and
        \item $\mts{X}_0\simeq X$.
    \end{itemize}
\end{defn}

\begin{defn}
Let $X$ be an $n$-dimensional $\mb{Q}$-Fano variety, and $E$ a prime divisor on a normal projective variety $Y$, where $\pi:Y\rightarrow X$ is a birational morphism. Then the \emph{log discrepancy} of $X$ with respect to $E$ is $$A_{X}(E):=1+\coeff_{E}(K_Y-\pi^{*}K_X).$$ We define the \emph{S-invariant} of $X$ with respect to $E$ to be $$S_{X}(E):=\frac{1}{(-K_X)^n}\int_{0}^{\infty}\vol_Y(-\pi^{*}K_X-tE)dt,$$ and the \emph{$\beta$-invariant} of $X$ with respect to $E$ to be $$\beta_{X}(E):=A_{X}(E)-S_{X}(E)$$
\end{defn}

\begin{theodef} \textup{(cf. \cite{Fuj19,Li17,BX19})} A $\mb{Q}$-Fano variety $X$ is 
\begin{enumerate}
    \item K-semistable if and only if $\beta_{X}(E)\geq 0$ for any prime divisor $E$ over $X$;
    \item K-stable if and only if $\beta_{X}(E)>0$ for any prime divisor $E$ over $X$;
    \item K-polystable if and only if it is K-semistable and any $\mb{G}_m$-equivariant K-semistable degeneration of $X$ is isomorphic to itself.
\end{enumerate}
A weak $\mb{Q}$-Fano variety $X$ is \textup{K-(semi/poly)stable} if its anti-canonical model $\ove{X}:=\Proj R(-K_X)$ is K-(semi/poly)stable.

\end{theodef}


Now we introduce the CM line bundle of a flat family of $\mb{Q}$-Fano varieties, which is a functorial line bundle over the base (cf. \cite{PT06,PT09,Tia97}). 

Let $\pi:\mts{X}\rightarrow S$ be a proper flat morphism of connected schemes with $S_2$ fibers of pure dimension $n$, and $\mtc{L}$ be an $\pi$-ample line bundle on $\mts{X}$. By \cite{KM76}, there are line bundles $\lambda_i=\lambda_i(\mts{X},\mtc{L})$ on $S$ such that $$\det(\pi_{!}(\mtc{L}^k))=\lambda_{n+1}^{\otimes\binom{k}{n+1}}\otimes \lambda_n^{\otimes\binom{k}{n}}\otimes \cdots\otimes\lambda_0^{\otimes\binom{0}{n}}$$ for any $k\gg0$. Write the Hilbert polynomial for each fiber $\mts{X}_s$ as $$\chi(\mts{X}_s,\mtc{L}^k_{s})=b_0k^n+b_1k^{n-1}+O(k^{n-2}).$$
 
\begin{defn}
    The \emph{CM $\mb{Q}$-line bundle} of the polarized family $(\pi:\mts{X}\rightarrow S,\mtc{L})$ are $$\lambda_{\CM,\pi,\mtc{L}}:=\lambda_{n+1}^{n(n+1)+\frac{2b_1}{b_0}}\otimes \lambda_n^{-2(n+1)}.$$
    If, in addition, both $\mts{X}$ and $S$ are normal, and $-K_{\mts{X}/S}$ is a $\pi$-ample $\bQ$-Cartier divisor, then we write
    $\lambda_{\CM,\pi} = l^{-n}\lambda_{\CM, \pi, \mtc{L}}$ where  $\mtc{L} = -l K_{\mts{X}/S}$ is a $\pi$-ample line bundle for some $l\in \bZ_{>0}$.
\end{defn}

The following result is essentially due to \cite{PT09}.

\begin{theorem}\label{KimpliesGIT} \textup{(cf. \cite[Thm.2.22]{ADL24})}
Let $f:\mts{X}\rightarrow S$ be a $\mb{Q}$-Gorenstein family of $\bQ$-Fano varieties over a normal projective base $S$. Let $G$ be a reductive group acting on $\mts{X}$ and $S$ such that $f$ is $G$-equivariant. Moreover, assume the following conditions are satisfied:
\begin{enumerate}[(i)]
\item for any $s\in S$, if $\Aut(\mts{X}_s)$ is finite, then the stabilizer $G_s$ is also finite;
\item if we have $\mts{X}_s\simeq \mts{X}_{s'}$ for $s,s'\in S$, then $s'\in G\cdot s$;
\item $\lambda_{\CM,f}$ is an ample $\mb{Q}$-line bundle on $S$.
\end{enumerate}
Then $s\in S$ is a GIT (poly/semi)stable point with respect to the $G$-linearized $\mb{Q}$-line bundle $\lambda_{\CM,f}$ if $\mts{X}_s$ is a K-(poly/semi)stable $\mb{Q}$-Fano variety.
\end{theorem}

The following theorem is usually called the \emph{K-moduli Theorem}, which is attributed to many people (cf. \cite{ABHLX20,BHLLX21,BLX19,BX19,CP21,Jia20,LWX21,LXZ22,Xu20,XZ20,XZ21}).

\begin{theorem} [K-moduli Theorem]
Fix two numerical invariants $n\in \bN$ and $V\in \bQ_{>0}$. Consider the moduli functor $\mts{M}^K_{n,V}$ sending a  base scheme $S$ to

\[
\left\{\mts{X}/S\left| \begin{array}{l} \mts{X}\to S\textrm{ is a proper flat morphism, each geometric fiber}\\ \textrm{$\mts{X}_{\bar{s}}$ is an $n$-dimensional K-semistable $\bQ$-Fano variety of}\\ \textrm{volume $V$, and $\mts{X}\to S$ satisfies Koll\'ar's condition}\end{array}\right.\right\}.
\]
Then there is an Artin stack, still denoted by $\mts{M}^K_{n,V}$, of finite type over $\mb{C}$ with affine diagonal which represents the moduli functor. The $\mb{C}$-points of $\mts{M}^K_{n,V}$ parameterize K-semistable  $\mb{Q}$-Fano varieties $X$ of dimension $n$ and volume $V$. Moreover, the Artin stack $\mts{M}^K_{n,V}$ admits a good moduli space $\ove{M}^K_{n,V}$, which is a projective scheme, whose $\mb{C}$-points parameterize K-polystable $\mb{Q}$-Fano varieties. The CM $\mb{Q}$-line bundle $\lambda_{\CM}$ on $\mts{M}^K_{n,V}$ descends to an ample $\mb{Q}$-line bundle $\Lambda_{\CM}$ on $\ove{M}^K_{n,V}$.
\end{theorem}

\begin{theorem}[cf. {\cite[Proposition 4.33]{CA21}}]\label{thm:general-Kstable}
A general smooth member of the family \textnumero 2.19 of Fano threefolds is K-stable.
\end{theorem}

\begin{defn}
Let $\mts{M}^K_{\textup{№2.19}}$ (resp. $\ove{M}^K_{\textup{№2.19}}$) be the irreducible component of $\mts{M}^K_{3,26}$ (resp. $\ove{M}^K_{3,26}$) with reduced stack (resp. scheme) structure whose general point parameterizes a K-stable blow-up of $\mb{P}^3$ along a smooth curve of genus $2$ and degree $5$, which is necessarily non-empty by Theorem \ref{thm:general-Kstable}. We call $\mts{M}^K_{\textup{№2.19}}$ (resp. $\ove{M}^K_{\textup{№2.19}}$) the K-moduli stack (resp. the K-moduli space) of the family \textnumero 2.19 of Fano threefolds.
\end{defn}

\begin{remark}\textup{
   By construction, a K-moduli stack $\mts{M}^K_{n,V}$ (resp. a K-moduli space $\ove{M}^K_{n,V}$) could be non-reduced, reducible or disconnected. See \cite{KP21, Pet21, Pet22}  for examples along this direction. In practice, an explicit K-moduli spaces we consider is often an irreducible component of $\ove{M}^K_{n,V}$ for some $n,V$, whose general point parametrizes a K-polystable (usually smooth) Fano variety  that we are interested in. As we shall see later in Corollary \ref{cor:stack-smooth}, the connected component of $\mts{M}^K_{3,26}$ containing the locus of K-stable smooth Fano threefolds of family \textnumero 2.19 is indeed a smooth irreducible stack. Thus a posteriori we could also define $\mts{M}^K_{\textup{№2.19}}$ as this connected component since smoothness holds.}
\end{remark}

\subsection{Moduli space of bundle stable pairs}

In this section, we collect several results on properties of Thaddeus' moduli spaces of bundle stable pairs. For more details, we refer the reader to \cite{Tha94}.

\begin{defn}\label{defn:thaddeus}
    Let $C$ be a smooth projective curve of genus $g\geq2$, $E$ be a vector bundle of rank $2$ over $C$, and $\phi\in H^0(C,E)$ be a non-zero section. Let $\sigma>0$ be a rational number. The pair $(E,\phi)$ is called \emph{$\sigma$-semistable} if for any line bundle $L\subseteq E$, one has \begin{equation}
        \deg L \ \leq\ \begin{cases}
            \frac{1}{2}\deg E-\sigma, \ \ \textup{if }\phi\in H^0(C,L) \\ 
            \\
            \frac{1}{2}\deg E+\sigma, \ \ \textup{if }\phi\notin H^0(C,L).
        \end{cases}
    \end{equation}
    It is \emph{$\sigma$-stable} if both inequalities are strict.
\end{defn}

\begin{theorem}
    Let $C$ be a smooth projective curve of genus $g$, and $\Lambda$ be a line bundle of degree $d=2g-1$ on $C$. Then there is a smooth projective good moduli space $\ove{M}_C(\sigma,\Lambda)$ parametrizing $\sigma$-polystable pairs $(E,\phi)$ such that $\det E=\Lambda$. The moduli space $\ove{M}_C(\sigma,\Lambda)$ is nonempty if and only if $\alpha\leq g-\frac{1}{2}$. Moreover, the followings hold.
    \begin{enumerate}
        \item For any fixed $i=0,...,g-1$, the moduli space $\ove{M}_C(\sigma,\Lambda)$ is independent of the choice of $\sigma\in\big(\max\{0,g-\frac{1}{2}-i-1\}\ ,\ g-\frac{1}{2}-i\ \big)$, and is denoted by $\ove{M}_{C,i}(\Lambda)$.
        \item (Boundary model I) There is a natural isomorphism $$\ove{M}_{C,0}(\Lambda)\ \simeq\ \bP H^1(\Lambda^{*})\ \simeq\ \bP^{3g-3},$$ under which $C$ is embedded into $\ove{M}_{C,0}(\Lambda)$ via the linear series $|K_C\otimes \Lambda|$.
        \item (Boundary model II) There is a natural surjective Abel-Jacobi map $$\ove{M}_{C,g-1}(\Lambda)\ \longrightarrow \ \ove{M}_C(2,\Lambda)$$ to the projective good moduli space of rank $2$ stable bundles of determinant $\Lambda$, whose fiber over $[E]$ is $\bP H^0(C,E)$.
        \item (Intermediate models) There is a divisorial contraction $\phi:\ove{M}_{C,1}(\Lambda)\rightarrow \ove{M}_{C,0}(\Lambda)$ which blows up $\ove{M}_{C,0}(\Lambda)$ along $C$. All the other intermediate moduli spaces are connected by flips as in the following diagram $$\xymatrix{
      & \wt{M}_{C,2}(\Lambda) \ar[dl] \ar[dr]  & &  \wt{M}_{C,3}(\Lambda) \ar[dl] \ar[dr]  & & \wt{M}_{C,g-1}(\Lambda) \ar[dl] \ar[dr] \\
  \ove{M}_{C,1}(\Lambda)  \ar[d]^{\phi}   &  &  \ove{M}_{C,2}(\Lambda) & & \cdots & &  \ove{M}_{C,g-1}(\Lambda)  \ \ \ \ar[d]  \\
   \ove{M}_{C,0}(\Lambda) \simeq \bP^{3g-3}  & & & & & & \ove{M}_C(2,\Lambda) \ \ \  }$$
    \end{enumerate}
\end{theorem}

\begin{theorem}[Picard group and ample cone] For any $1\leq i\leq g-1$, the Picard group of each moduli space $\ove{M}_{C,i}(\Lambda)$ is free of rank two, generated by the strict transform of $H_i:=\phi^*\mtc{O}_{\bP^{3g-3}}(1)$ and the strict transform $E_i$ of the $\phi$-exceptional divisor. Moreover, if we denote $(m+n)H_i-nE_i$ by $\mtc{O}_i(m,n)$, then the ample cone of $\ove{M}_{C,g-1}(\Lambda)$ is generated by $\mtc{O}_{g-1}(1,g-2)$ and $\mtc{O}_{g-1}(2,2g-3)$.
    
\end{theorem}

\begin{corollary}
    The moduli space $\ove{M}_{C,g-1}(\Lambda)$ is a smooth Fano variety. 
\end{corollary}

\begin{remark}
    \textup{In fact, among all the moduli spaces $\ove{M}_{C,i}(\Lambda)$ for $1\leq i\leq g-1$, $\ove{M}_{C,g-1}(\Lambda)$  is the unique one which is Fano; see \cite[(5.3)]{Tha94}.}
\end{remark}

\begin{example}\textup{
    Let $C$ be a smooth projective curve of genus $2$. For any line bundle $\Lambda$ on $C$ of degree $3$, one can embed $C$ into $$\bP H^0(C,K_C\otimes\Lambda)^{*}\ \simeq \ \bP H^1(C,\Lambda^*)\ \simeq\  \bP^3$$ via the complete linear series $|K_C\otimes\Lambda|$. The moduli space $N:=\ove{N}_{C}(2,\Lambda)$ of rank two vector bundles with determinant $\Lambda$ is a del Pezzo threefold of volume $32$, i.e. a $(2,2)$-complete intersection in $\bP^5$. The Brill--Noether locus $B$ of $N$ which parametrizes vector bundles $E$ with $h^0(E)\geq 2$ is a line, i.e. a smooth rational curve of degree $1$ in $\bP^5$. In this example, there is only one non-trivial moduli space of bundle stable pairs $$\ove{M}_{C,1}(\Lambda)\ \simeq \ \Bl_{C}\bP H^1(C,\Lambda^{*}) \ \simeq \ \Bl_{B}N.$$ The wall-crossing structure is exactly the Sarkisov link of the one associated to Fano family №2.19.
}\end{example}

\section{Moduli spaces of K3 surfaces}\label{K3SURFACES}

In this section, we prove some results on the geometry of moduli spaces of (lattice-) polarized K3 surfaces. This is useful in understanding the anti-canonical linear series of a (weak) Fano variety.

\subsection{Geometry and moduli of K3 surfaces}

\begin{defn}
    Let $L_{K3}:=\mb{U}^{\oplus3}\oplus \mb{E}_8^{\oplus2}$ be a fixed (unique) even unimodular lattice of signature $(3,19)$.
\end{defn}

Let $\Lambda$ be a rank $r$ primitive sublattice of $L_{K3}$ with signature $(1,r-1)$. 
A vector $h\in \Lambda\otimes \mb{R}$ is called \emph{very irrational} if $h\notin \Lambda'\otimes \mb{R}$ for any primitive proper sublattice $\Lambda'\subsetneq \Lambda$. Fix a very irrational vector $h$ with $(h^2)>0$.

\begin{defn}
    A \emph{$\Lambda$-polarized K3 surface} (resp. a \emph{$\Lambda$-quasi-polarized K3 surface}) $(X,j)$ is a K3 surface $X$ with ADE singularities (resp. a smooth projective surface $X$) together with a primitive lattice embedding $j:\Lambda\hookrightarrow\Pic(X)$ such that $j(h)\in \Pic(X)_{\mb{R}}$ is ample (resp. big and nef).
    \begin{enumerate}
        \item Two such pairs $(X_1,j_1)$ and $(X_2,j_2)$ are called \emph{isomorphic} if there is an isomorphism $f:X_1\stackrel{\simeq}{\rightarrow} X_2$ of K3 surfaces such that $j_1=f^{*}\circ j_2$.
        \item The \emph{$\Lambda$-(quasi-) polarized period domain} is $$\mb{D}_{\Lambda}:=\mb{P}\{w\in {\Lambda}^{\perp}\otimes\mb{C}:(w^2)=0,\ (w.\ove{w})>0\}.$$
    \end{enumerate}
    When $r=1$, i.e. $\Lambda$ is of rank one, it is convenient to choose $h$ to be the effective generator $L$ of $\Lambda$. We denote by $d$ the self-intersection of $L$, and we call $(X,L)$ a (quasi-)polarized K3 surface of degree $d$.
\end{defn}



\begin{defn}
   For a fixed lattice $\Lambda$ with a very irrational vector $h$, one define the \emph{moduli functor $\mts{F}_{\Lambda}$ of $\Lambda$-polarized K3 surfaces} to send a base scheme $T$ to 
\[
\left\{(f:\mts{X}\rightarrow T;\varphi)\left| \begin{array}{l} \mts{X}\to T\textrm{ is a proper flat morphism, each geometric fiber}\\ \textrm{$\mts{X}_{\bar{t}}$ is an ADE K3 surface, and $\varphi:\Lambda\longrightarrow\Pic_{\mts{X}/T}(T)$ is }\\ \textrm{a group homomorphism such that the induced map }\\ \textrm{$\varphi_{\bar{t}}:\Lambda\rightarrow \Pic(\mts{X}_{\bar{t}})$ is an isometric primitive embedding of}\\ \textrm{lattices and that $\varphi_{\bar{t}}(h)\in \Pic(\mts{X}_{\bar{t}})_{\mb{R}}$ is an ample class.} \end{array}\right.\right\}.
\]

\end{defn}


\begin{theorem}[cf. \cite{Dol96, AE23}]\label{isommoduli}
    The moduli functor of $\Lambda$-polarized K3 surfaces is represented by a smooth separated Deligne-Mumford (DM) stack $\mts{F}_{\Lambda}$. 
    Moreover, $\mts{F}_{\Lambda}$ admits a coarse moduli space $F_{\Lambda}$, whose analytification is isomorphic to $\mb{D}_{\Lambda}/\Gamma$, where $\Gamma:=\{\gamma\in \mathrm{O}(L_{K3}):\gamma|_{\Lambda}=\Id_{\Lambda} \}$.
\end{theorem}

\begin{remark}
    \textup{When $\Lambda$ is of rank one, we denote by $d$ the self-intersection of a generator of $\Lambda$, and by $\mts{F}_d$ (resp. $F_d$) the corresponding moduli stack (resp. coarse moduli space).}
\end{remark}

\subsection{K3 surfaces in anticanonical linear series}\label{k33}

Let $\mathscr{H}_{8}\subseteq \mts{F}_8$ be the Noether-Lefschetz divisor in the moduli stack of polarized K3 surfaces of degree $8$ parametrizing those K3 surfaces which contain a line, and $\mathscr{H}_{8}^{\nu}\rightarrow \mathscr{H}_{8}$ be the normalization. Here, a \emph{line} on $(S,L)\in \mts{F}_8$ is a connected smooth rational curve $\ell$ with $(\ell,L)=1$.
Let $\Lambda_0$ be a rank $2$ hyperbolic sublattice with generator $L,E$ satisfying that $$(L^2)=8,\quad (L.E)=1,\quad \textup{and}\quad (E^2)=-2,$$ and signature $(1,1)$. Let $J:\Lambda_0\hookrightarrow L_{K3}$ be a primitive embedding. Let $\mts{F}_{\Lambda_1}$ (resp. $\mts{F}_{\Lambda_2}$) be the moduli space of $\Lambda_0$-polarized K3 surface with respect to the very irrational vector $h=h_1:=L-\epsilon E$ (resp. $h=h_2:=2L-(1-\epsilon) E$) where $0<\epsilon \ll 1$ is an irrational number. By \cite[Section 2 of arXiv version 1]{AE23}, for each $i\in \{1,2\}$ the universal K3 surface over the moduli stack $\mts{F}_{\Lambda_i}$ is independent of the choice of $\epsilon$. As we shall see later,  these two universal families are indeed isomorphic by Lemma \ref{samefamily}.

\begin{prop}\label{isomo}
    There is an isomorphism of moduli stacks $$\mts{F}_{\Lambda_1}\stackrel{\simeq}{\longrightarrow} \mathscr{H}_{8}^{\nu}.$$ Moreover, the total space of the universal family over $\mts{F}_{\Lambda_1}$ is the blow-up of the universal family over $\mts{H}_{8}^{\nu}$ along the universal line. 
\end{prop}

\begin{proof}
   The proof follows by the same reasoning as in \cite[Proposition 3.6]{LZ24}.
\end{proof}

\begin{lemma}\label{lemma:nounigonal}
    Let $(S,L)$ be a unigonal polarized K3 surface of degree $8$. Then there is no line contained in $S$.
\end{lemma}

\begin{proof}
    As $(S,L)$ is unigonal, then one has that $L=E+5F$, where $E$ is a rational curve and $F$ is the class of a genus one curve. The linear series $|F|$ induces a elliptic fibration $f:S\rightarrow \mb{P}^1$, and $E$ is an $f$-section. It is easy to see that $\ell$ is contained in a (reducible) fiber of $f$, denoted by $F_0$. Writing $F_0=\ell+F_0'$, one has that $$(F_0'.L)\ = \ (F_0-\ell.L)\ =\ 1-1 \ = \ 0,$$ which contradicts with the ampleness of $L$.
\end{proof}

\begin{lemma}\label{ampleness}
    Let $(S,L)$ be a polarized K3 surface of degree $8$ which contains a line $\ell$. Let $\mu:\wt{S}:=\Bl_{\ell}S \rightarrow S$ be the blow-up of $S$ along $\ell$ with the exceptional divisor\footnote{Here, by \emph{exceptional divisor}, we mean that the Cartier divisor defined by the principal ideal $\mu^{-1}I_{\ell}\cdot \mtc{O}_{\wt{S}}$. It is not the exceptional divisor in the sense of birational geometry.} $E$. Then $2\mu^{*}L-E$ is nef. In particular, for any $0<\epsilon\ll1$, the $\mb{R}$-Cartier $\mb{R}$-divisor $2\mu^{*}L-(1-\epsilon)E$ is ample.
\end{lemma}

\begin{proof}
    We will show that the linear series $|2\mu^{*}L-E|$ does not have base component which intersects negatively with $2\mu^{*}L-E$. By Lemma \ref{lemma:nounigonal}, the polarized K3 surface $(S,L)$ is either of generic case or hyperelliptic case. In the former case, $S$ is embedded by $|L|$ into $\bP^5$. It is easy to see that $\mtc{I}_{\ell/S}(2L)$ is globally generated since $\mtc{I}_{\ell/\bP^5}(2)$ is, and hence any base component of $|2\mu^{*}L-E|$ (if exists) is $\mu$-exceptional, on which $2\mu^{*}L-E$ is positive.
    
    If $(S,L)$ is hyperelliptic, let $f:S\rightarrow R\subseteq \mb{P}^5$ be the double cover induced by the linear series $|L|$. The image of $\ell$, denoted by $\ell_0$, is a (reduced) line in $\mb{P}^5$ as $$ (\ell_0.\mtc{O}_{\mb{P}^5}(1)) \ =\ (f_{*}\ell.\mtc{O}_{\mb{P}^5}(1)) \ =\ (\ell.f^{*}\mtc{O}_{\mb{P}^5}(1))\ =\ (\ell.L)\ = \ 1.$$ Since $f:S\rightarrow R$ is a double cover, then there is another line $\ell'$ (not necessarily distinct from $\ell$) on $X$ such that $f^{*}\ell_0=\ell+\ell'$. Since $\mtc{I}_{\ell_0/\bP^5}(2)$ is globally generated, then the only possible non-$\mu$-exceptional base component of $|2\mu^*L-E|$ is the strict transform $\wt{\ell}'$ of $\ell'$. However, one has that $$(2\mu^{*}L-E.\wt{\ell}')  \ \geq \ (2L.l')-(l.l')\ \geq\ 2-1 \ = \ 1.$$ On the other hand, notice that $2\mu^{*}L-E$ is $\mu$-ample, and hence $2\mu^{*}L-Q$ is nef.

\end{proof}

\begin{lemma}\label{samefamily}
    The moduli stacks $\mts{F}_{\Lambda_1}$ and $\mts{F}_{\Lambda_2}$ are isomorphic. Moreover, their universal K3 surfaces are isomorphic. 
\end{lemma}

\begin{proof}
    The same argument as in \cite[Lemma 3.8]{LZ24} applies.
\end{proof}

\begin{remark}\textup{The results on moduli of K3 surfaces will be used in Section \ref{limitkss}: to show that the K-semistable limit is the blow-up of $\mb{P}^3$ along a genus two quintic curve, we in fact show that it is the blow-up of a quartic del Pezzo threefold along a line. When analyzing the behavior of the linear series on the limiting threefold, we need to restrict it first to the anticanonical K3 surfaces.}
\end{remark}

\begin{remark}\textup{
    By Lemma \ref{samefamily}, we can write $\mts{F}_{\Lambda_0}$ to be $\mts{F}_{\Lambda_1}\simeq \mts{F}_{\Lambda_2}$, and the universal family $u:\mts{S}_{\Lambda_0}\rightarrow \mts{F}_{\Lambda_0}$.}
\end{remark}

\begin{remark}
\textup{Indeed, it is shown in \cite[Section 2 of arXiv version 1]{AE23} that for any fixed lattice $\Lambda_0$ and two  very irrational vector $h_1, h_2\in \Lambda_0\otimes \bR$, the moduli stack $\mts{F}_{\Lambda_1}$ and $\mts{F}_{\Lambda_2}$ are isomorphic, where $\Lambda_i$ denotes the lattice $\Lambda_0$ with polarization $h_i$. However, their universal K3 surfaces may not be isomorphic but only birational on each fiber. If, in addition, $h_1$ and $h_2$ belong to the same so-called \emph{small cone} (see \cite[Section 2 of arXiv version 1]{AE23}), then their universal K3 surfaces are isomorphic. Lemma \ref{samefamily} shows that $h_1=L-\epsilon E$ and $h_2 = 2L-(1-\epsilon)E$ lie in the same small cone in our case.}
\end{remark}

\section{K-semistable limits of a one-parameter family}\label{KSSLIMITS}

In this section, we will study the K-semistable $\mb{Q}$-Gorenstein limits of the Fano threefolds in the deformation family №2.19. Our goal is to prove the following characterization of such limits.

\begin{theorem}\label{23complete}
Let $X$ be a K-semistable $\bQ$-Fano variety that admits a $\bQ$-Gorenstein smoothing to smooth Fano threefolds in the family \textnumero 2.19. Then  $X$ is Gorenstein canonical and isomorphic to the blow-up of a $(2,2)$-complete intersection $V$ in $\bP^5$ along a line which is not contained in the singular locus of $V$.
\end{theorem}

\subsection{Volume comparison and general elephants}

In this subsection, we use local-to-global volume comparison from \cite{Fuj18, Liu18, LX19, Liu22} and general elephants on (weak) Fano threefolds  \cite{Sho79, Rei83} (see also \cite{Amb99, Kaw00}) to show that every K-semistable limit $X$ of family \textnumero 2.19 is Gorenstein canonical whose anti-canonical linear system $|-K_X|$ is base-point-free.

\begin{theorem}[ref. \cite{LZ24}]\label{thm:nonvanishing}
Let $X$ be a $\bQ$-Gorenstein smoothable K-semistable (weak) $\mb{Q}$-Fano threefold with volume $V:=(-K_X)^3\geq 20$. Then the followings hold.
\begin{enumerate}
    \item The variety $X$ is Gorenstein canonical;
    \item There exists a divisor $S\in |-K_X|$ such that $(X,S)$ is a plt pair, and that $(S,-K_X|_S)$ is a (quasi-) polarized K3 surface of degree $V$.
    \item If $D$ is a $\mb{Q}$-Cartier Weil divisor on $X$ which deforms to a $\bQ$-Cartier Weil divisor on a $\bQ$-Gorenstein smoothing of $X$, then $D$ is Cartier. 
\end{enumerate}
\end{theorem}

\begin{theorem}\label{generalelephant}\textup{(General elephants, cf. \cite{Rei83,Sho79})}
    Let $X$ be a Gorenstein canonical weak Fano threefold. Then $|-K_X|\neq \emptyset$, and a general element $S\in|-K_X|$ is a K3 surface with at worst du Val singularities.
\end{theorem}

\begin{lemma}\label{deform}
    Let $X$ be a Gorenstein canonical (weak) Fano threefold and $\pi:\mts{X}\rightarrow T$ be a $\mb{Q}$-Gorenstein smoothing of $X=\mts{X}_0$. Then for any K3 surface $S\in|-K_{X}|$ \textup{(}e.g. $S$ is a general member in $|-K_{X}|$\textup{)}, after possibly shrinking $T$, there exists a family of K3 surface $\mts{S}\subseteq \mts{X}$ over $T$ such that $\mts{S}_t\in|-K_{\mts{X}_t}|$ and $\mts{S}_0=S$. In particular, $\mts{S}$ is a Cartier divisor in $\mts{X}$.
\end{lemma}

\begin{proof}
    Consider the exact sequence $$0\longrightarrow \mtc{O}_{\mts{X}}(-K_{\mts{X}}-\mts{X}_0)\longrightarrow \mtc{O}_{\mts{X}}(-K_{\mts{X}})\longrightarrow \mtc{O}_{\mts{X}_0}(-K_{\mts{X}_0})\longrightarrow 0.$$ Applying $\pi_{*}$, one obtains by the Kawamata-Viehweg vanishing theorem that the restriction $$H^0(\mts{X},\mtc{O}_{\mts{X}}(-K_{\mts{X}}))\longrightarrow H^0(\mts{X}_0,\mtc{O}_{\mts{X}_0}(-K_{\mts{X}_0}))$$ is surjective. Moreover, if $S$ is a general member in $|-K_{\mts{X}_0}|$, then one can choose $\mts{S}$ to be general, and hence $\mts{S}_t\in |-K_{\mts{X}_t}|$ is general for any general point $t\in T$.
    
\end{proof}

\subsection{Modifications on the family and degree 8 K3 surfaces}\label{limitkss}

From now on, let $X$ be a K-semistable $\bQ$-Fano variety that admits a $\bQ$-Gorenstein smoothing $\pi:\mts{X}\rightarrow T$  over a smooth pointed curve $0\in T$ such that $\mts{X}_0\simeq X$ and every fiber $\mts{X}_t$ over $t\in T\setminus \{0\}$ is a smooth Fano threefold in the family №2.19. Up to a finite base change, we may assume that the restricted family $\mts{X}^{\circ}\rightarrow T^{\circ}:=T\setminus \{0\}$ is isomorphic to $\Bl_{\mts{C}^{\circ}}(\mb{P}^3\times T^{\circ})$, where $\mts{C}^{\circ}\rightarrow T^{\circ}$ is a family of smooth genus two curves in $\mb{P}^3$ of degree five. Let $\mts{E}^{\circ}\subseteq \mts{X}^{\circ}$ be the Cartier divisor corresponding to the exceptional divisor of the blow-up, and $\wt{\mts{Q}}^{\circ}$ be the strict transform of the family of quadric surfaces $\mts{Q}^{\circ}$ containing $\mts{C}^{\circ}$. Let $\mtc{H}^{\circ}$ be the line bundle on $\mts{X}^{\circ}$ which is the pull-back of $\mtc{O}_{\mb{P}^3}(1)$, and $\mtc{L}^{\circ}:=3\mtc{H}^{\circ}-\mts{E}^{\circ}$. Then $\mts{E}^{\circ}$ (resp. $\mtc{H}^{\circ}$, $\mtc{L}^{\circ}$, $\mts{Q}^{\circ}$) extends to $\mts{E}$ (resp. $\mtc{H}$, $\mtc{L}$, $\mts{Q}$) on $\mts{X}$ as a Weil divisor on $\mts{X}$ by taking Zariski closure. We denote by $L$ (resp. $H$, $E$, $Q$) the restriction of $\mtc{L}$ (resp. $\mtc{H}$, $\mts{E}$, $\mts{Q}$) on the central fiber $\mts{X}_0\simeq X$.

Let $\theta:\mts{Y}\rightarrow \mts{X}$ be a small $\mb{Q}$-factorialization of $\mts{X}$. Since $\mts{X}$ is klt and $K_{\mts{Y}}=\theta^* K_{\mts{X}}$, we know that $\mts{Y}$ is $\bQ$-factorial of Fano type over $\mts{X}$. By \cite{BCHM} we can run a minimal model program (MMP) for $\theta^{-1}_{*}\mts{Q}$ on $\mts{Y}$ over $\mts{X}$. As a result, we obtain a log canonical model $\mts{Y} \dashrightarrow \wt{\mts{X}}$ that fits into a commutative diagram
$$\xymatrix{
 & \wt{\mts{X}} \ar[rr]^{f} \ar[dr]_{\wt{\pi}}  &  &  \mts{X} \ar[dl]^{\pi}\\
 & & T  &\\
 }$$ satisfying the following conditions:
\begin{enumerate}
    \item $f$ is a small contraction, and is an isomorphism over $T^{\circ}$;
    \item $-K_{\wt{\mts{X}}}=f^{*}(-K_{\mts{X}})$ is a  $\wt{\pi}$-big and $\wt{\pi}$-nef Cartier divisor; 
    \item $\wt{\mts{X}}_0$ is a $\bQ$-Gorenstein smoothable Gorenstein canonical weak Fano variety whose anti-canonical model is isomorphic to $\mts{X}_0$;
    \item $\wt{\mts{Q}}:=f^{-1}_* \mts{Q}$ is an $f$-ample Cartier Weil divisor (cf. Theorem \ref{thm:nonvanishing}(3)); and
    
    \item $\wt{\mtc{L}}:=f^{-1}_*\mtc{L}\sim_{\bQ, T}\frac{1}{2}(-K_{\wt{\mtc{X}}} + \wt{\mts{Q}})$ is a Cartier divisor (cf. Theorem \ref{thm:nonvanishing}(3)), and $2\wt{\mtc{L}}-(1-\epsilon)\wt{\mts{Q}}\sim_{\bR, T} -K_{\wt{\mtc{X}}} + \epsilon \wt{\mts{Q}}$ is $\wt{\pi}$-ample, for any real number $0<\epsilon\ll1$.
\end{enumerate}

To ease our notation, we denote by $$\wt{X}:=\wt{\mts{X}}_0,\ \ \  g = f|_{\wt{X}}: \wt{X} \to X,\ \ \ \wt{\mtc{H}}:=f_{*}^{-1} \mtc{H},\ \ \ \textup{and}\ \ \ \wt{\mts{E}}:=f_{*}^{-1} \mts{E}.$$ By Theorem \ref{thm:nonvanishing}(3) and linear equivalences $$\wt{\mtc{H}}\sim_{T} \wt{\mtc{L}} - \wt{\mts{Q}}\ \ \ \textup{and} \ \ \ \wt{\mts{E}}\sim_{T} 2\wt{\mtc{L}} - 3\wt{\mts{Q}},$$ we know that both $\wt{\mtc{H}}$ and $\wt{\mts{E}}$ are Cartier. We also denote by $\wt{L}$ (resp. $\wt{H}$, $\wt{E}$, $\wt{Q}$) the restriction of $\wt{\mtc{L}}$ (resp. $\wt{\mtc{H}}$, $\wt{\mts{E}}$, $\wt{\mts{Q}}$) to $\wt{\mts{X}}_0 = \wt{X}$. 

In the rest of this section, our goal is to show that $\wt{\mtc{L}}$ is relatively big and semiample over $T$ (see Proposition \ref{nefness}). A key ingredient is our previous study on the moduli of certain lattice-polarized K3 surfaces (see Lemma \ref{samefamily}).

\begin{lemma}\label{quasipolarized}
If $\wt{S}\in|-K_{\wt{X}}|$ is a K3 surface (e.g. when $S$ is a general member), then $(\wt{S},\wt{L}|_{\wt{S}})$ is a quasi-polarized degree $8$ K3 surface.
\end{lemma}

\begin{proof}

For any $0<\epsilon \ll 1$, since $\wt{\mts{Q}}$ is $f$-ample, then we know that $2\wt{\mtc{L}}-(1-\epsilon)\wt{\mts{Q}}$ is $\wt{\pi}$-ample, and hence $(2\wt{L}-(1-\epsilon)\wt{Q})|_{\wt{S}}$ is ample on the K3 surface $\wt{S}$. 

We claim that $\wt{S}$ is represented by a point in $\mts{F}_{\Lambda_2}$ (cf. Section \ref{k33}). Indeed, by deforming to a family of K3 surface in $|-K_{\mts{X}_t}|$ (cf. Lemma \ref{deform}), one sees that the divisors $\wt{H}|_{\wt{S}}\sim (\wt{L}-\wt{Q})|_{\wt{S}}$ and $\wt{E}|_{\wt{S}}\sim (2\wt{L}-3\wt{Q})|_{\wt{S}}$ generate a primitive sublattice of $\Pic(\wt{S})$ which is isometric to $\Lambda_2$ where $(2\wt{L}-(1-\epsilon)\wt{Q})|_{\wt{S}}$ corresponds to the vector $h_2$.

By Lemma \ref{samefamily}, we know that $\wt{S}$ is represented by a point in $\mts{F}_{\Lambda_1}$ as well. In particular, $(\wt{L}-\epsilon\wt{Q})|_{\wt{S}}$ is also ample for any $0<\epsilon \ll 1$, and hence $\wt{L}|_{\wt{S}}$ is nef. The degree of $\wt{L}|_{\wt{S}}$ is invariant under deformation, as $\wt{\mtc{L}}$ is a Cartier divisor on $\wt{\mts{X}}$. 
    
\end{proof}

\begin{lemma}\label{isomorphismonsection}
    For any K3 surface $\wt{S}\in |-K_{\wt{X}}|$, the restriction map $$H^0(\wt{X},\mtc{O}_{\wt{X}}({\wt{L}}))\longrightarrow H^0(\wt{S},\mtc{O}_{\wt{S}}({\wt{L}|_{\wt{S}}}))$$ is an isomorphism. In particular, we have that $h^0(\wt{X},\mtc{O}_{\wt{X}}({\wt{L}}))=6$. 
\end{lemma}

\begin{proof}
    Let $\wt{S}\in|-K_{\wt{X}}|$ be a K3 surface. By Lemma \ref{quasipolarized} we see that $(\wt{S},\wt{L}|_{\wt{S}})$ is a quasi-polarized degree $8$ K3 surface, and hence $$h^0(\wt{S},\wt{L}|_{\wt{S}})\ =\ \frac{1}{2}({\wt{L}}|_{\wt{S}})^2+2\ =\ 6.$$
    Since $\wt{S}\sim -K_{\wt{X}}$ is Cartier, we have a short exact sequence 
    \[
    0 \to \mtc{O}_{\wt{X}}({\wt{L}}- \wt{S}) \to \mtc{O}_{\wt{X}}({\wt{L}}) \to \mtc{O}_{\wt{S}}({\wt{L}}|_{\wt{S}})\to 0.
    \]
    Notice that ${\wt{L}}-\wt{S}\sim \wt{L}+K_{\wt{X}}\sim -\wt{H}$ is not effective. Hence by taking the long exact sequence we see that 
    $H^0(\wt{X},\mtc{O}_{\wt{X}}({\wt{L}}))\hookrightarrow H^0(\wt{S},{\wt{L}}|_{\wt{S}})$ is injective, and thus $h^0(\wt{X},\mtc{O}_{\wt{X}}({\wt{L}}))\leq 6$. On the other hand, by upper semi-continuity, we have that $$h^0(\wt{X},\mtc{O}_{\wt{X}}({\wt{L}}))\ \geq\  h^0(\wt{\mts{X}}_t,\mtc{O}_{\wt{\mts{X}}_t}({\wt{\mtc{L}}_t}))\ =\ 6$$ for a general $t\in T$. Therefore, one has $h^0(\wt{X},\mtc{O}_{\wt{X}}({\wt{L}}))=6$, and the restriction map is an isomorphism.
    
\end{proof}

\begin{lemma}\label{except}
    We have that $h^0\big(\wt{X},\mtc{O}_{\wt{X}}(\wt{Q})\big)=1$.
\end{lemma}

\begin{proof}
    It suffices to show that $h^0(\wt{X},\mtc{O}_{\wt{X}}(\wt{Q}))\leq1$ as $\wt{Q}$ is effective. For a general K3 surface $\wt{S}\in |-K_{\wt{X}}|$, $\wt{L}|_{\wt{S}}$ is big and nef. Moreover, by taking the ample model, denoted by $\ove{S}$, of $\wt{S}$ with respect to $\wt{L}|_{\wt{S}}$, one get a birational morphism which sends $\wt{Q}|_{\wt{S}}$ to a line $l$. Thus $$h^0(\wt{S},\mtc{O}_{\wt{S}}(\wt{Q}))\ = \ h^0(\ove{S},\mtc{O}_{\ove{S}}(l))\ =\ 1.$$ As $\wt{Q}-\wt{S}\sim -2\wt{H}$ is not effective, one has an injection  $H^0(\wt{X},\mtc{O}_{\wt{X}}(\wt{Q}))\hookrightarrow H^0(\wt{S},\mtc{O}_{\wt{S}}(\wt{Q}))$, which proves the statement.

\end{proof}

\begin{lemma}\label{isomorphismonsection2}
    For any K3 surface $\wt{S}\in |-K_{\wt{X}}|$, The restriction map $$H^0(\wt{X},\mtc{O}_{\wt{X}}(2{\wt{L}}))\longrightarrow H^0(\wt{S},\mtc{O}_{\wt{S}}(2{\wt{L}}|_{\wt{S}}))$$ is surjective. In particular, we have that $h^0(\wt{X},\mtc{O}_{\wt{X}}(2{\wt{L}}))=19$. 
\end{lemma}

\begin{proof}
    Let $\wt{S}\in|-K_{\wt{X}}|$ be a K3 surface. By Lemma \ref{quasipolarized} we see that $(\wt{S},{\wt{L}}|_{\wt{S}})$ is a  quasi-polarized K3 surface of degree $6$, and hence $$h^0(\wt{S},2{\wt{L}}|_{\wt{S}})\ =\ \frac{1}{2}(2{\wt{L}}|_{\wt{S}})^2+2\ =\ 18.$$
    Arguing as in the proof of Lemma \ref{isomorphismonsection}, we have a long exact sequence
    \[
    0 \to H^0(\wt{X}, \mtc{O}_{\wt{X}}(2{\wt{L}} - \wt{S})) \to H^0(\wt{X}, \mtc{O}_{\wt{X}}(2{\wt{L}})) \to H^0(\wt{S}, \mtc{O}_{\wt{S}}(2{\wt{L}|_{\wt{S}}})).
    \]
    Notice that $2{\wt{L}}-\wt{S}\sim 2{\wt{L}}+K_{\wt{X}}\sim \wt{Q}$ and that $h^0(\wt{X},\mtc{O}_{\wt{X}}(\wt{Q}))=1$ (cf. Lemma \ref{except}). Hence the above long exact sequence implies that $$h^0(\wt{X},\mtc{O}_{\wt{X}}(2{\wt{L}}))\ \leq\  h^0(\wt{S},2{\wt{L}}|_{\wt{S}})+1\ =\ 19.$$ On the other hand, by upper semi-continuity, we have that $$h^0(\wt{X},\mtc{O}_{\wt{X}}(2{\wt{L}}))\ \geq\  h^0(\wt{\mts{X}}_t,\mtc{O}_{\wt{\mts{X}}_t}(2{\wt{\mtc{L}}_t}))\ =\ 19$$ for any general $t\in T$. Thus the proof is finished.
    
\end{proof}

\begin{prop}\label{nefness} The Cartier divisor $\wt{\mtc{L}}$ is $\wt{\pi}$-semiample and  $\wt{\pi}$-big.

\end{prop}

\begin{proof}
    We first prove that $\wt{L}$ is a nef divisor. Recall that we have $2\wt{L}\sim -K_{\wt{X}}+\wt{Q}$. As $-K_{\wt{X}}$ is big and nef, and $\wt{Q}$ is $g$-ample, then we have that $(\wt{L}.C)>0$ for any $g$-exceptional curve $C\subseteq \wt{X}$.

    We claim that the base locus of the linear series $|2\wt{L}|$ is either some isolated points, or is contained in the $g$-exceptional locus, so $\wt{L}$ is nef. By Lemma \ref{quasipolarized}, we know that $|2\wt{L}|_{\wt{S}}|$ is base-point-free, where $\wt{S}\in|-K_{\wt{X}}|$ is a general member. Suppose that $\wt{C}\subseteq \Bs|2\wt{L}|$ is a curve which is not contracted by $g$. Then the intersection $\wt{C}\cap \wt{S}$ is non-empty and consists of finitely many points, which are all base points of $|2\wt{L}|_{\wt{S}}|$ (cf. Lemma \ref{isomorphismonsection2}). This leads to a contradiction.

     Since $\wt{L}= \wt{\mtc{L}}|_{\wt{\mts{X}}_0}$ is nef, and $\wt{\mtc{L}}|_{\wt{\mts{X}}_t}$ is nef for any $t\in T\setminus\{0\}$ as $\wt{\mts{X}}_t\simeq \mts{X}_t$ is a smooth Fano threefold in the family \textnumero 2.19, we conclude that $\wt{\mtc{L}}$ is $\wt{\pi}$-nef. This implies the $\wt{\pi}$-semiampleness of $\wt{\mtc{L}}$ by Kawamata--Shokurov base-point-free theorem, as $\wt{\mts{X}}$ is of Fano type over $T$. Since $\wt{\mtc{L}}|_{\wt{\mts{X}}_t}$ is big for a general $t$, we get the $\wt{\pi}$-bigness of $\wt{\mtc{L}}$.
    
\end{proof}

\subsection{Birational model as a complete intersection of two quadrics}

From the last subsection we know that $\wt{\mtc{L}}$ is a  $\wt{\pi}$-semiample and  $\wt{\pi}$-big Cartier divisor on $\wt{\mts{X}}$
(cf. Proposition \ref{nefness}). Thus taking its ample model over $T$ yields a birational morphism $\phi:\wt{\mts{X}}\to \mts{V}$ that fits into a commutative diagram $$\xymatrix{
 & \wt{\mts{X}} \ar[rr]^{\phi \quad} \ar[dr]_{\wt{\pi}}  &  &  \mts{V}  
 \ar[dl]^{\pi_{\mts{V}}} \\
 & & T&\\
 }$$ 
 Here we have $$\mts{V} \ :=\ \Proj_{T}\bigg(\bigoplus_{m\in \bN} \wt{\pi}_* \left(\wt{\mtc{L}}^{\otimes m}\right)\bigg).$$
 For any $0\neq t\in T$, the morphism $\wt{\mts{X}}_t\rightarrow \mts{V}_t$ contracts precisely the smooth quadric surface $\wt{\mts{Q}}_t$ to a line contained in $\mts{V}_t$, which can be embedded into $\mb{P}^5$ as a complete intersection of two quadrics by the line bundle $(\phi_{*}\mtc{L})|_{\mts{V}_t}$.

Now let us consider the restriction of the morphism $\phi$ to the central fiber $$\phi_0:\wt{\mts{X}}_0=\wt{X}\longrightarrow V:=\mts{V}_0.$$ Let $L_V:=(\phi_{0})_*\wt{L}$ be the $\mb{Q}$-Cartier Weil divisor on $V$. Our main goal in this subsection is to show that $V$ is also a complete intersection of two quadrics, and $\wt{X}\simeq X$ is the blow-up of $V$ along a line. These results will help us prove Theorem \ref{23complete}.

\begin{lemma}\label{lem:contract-Q}
    The central fiber $V$ of $\pi_{\mts{V}}$ is a normal projective variety. Moreover, 
    the morphism $\phi_0: \wt{X}\to V$ is birational, contracts $\wt{Q}$ to a line $\ell_V$ of $V$, and is an isomorphism on $\wt{X}\setminus \wt{Q}$.
\end{lemma}

\begin{proof}
    We first show that $V$ is normal and  $\phi_0$ is birational. Since both $\wt{\mtc{L}}$ and $-K_{\wt{\mts{X}}}$ are nef and big over $T$ and $\wt{\mts{X}}$ is klt, then Kawamata--Viehweg vanishing theorem implies that $$R^i \wt{\pi}_* \wt{\mtc{L}}^{\otimes m} \ =\  0$$ for any $i> 0$ and $m\in \bN$. Thus, by cohomology and base change, the sheaf $\wt{\pi}_* \wt{\mtc{L}}^{\otimes m}$ is locally free and satisfies that $$\left(\wt{\pi}_* \wt{\mtc{L}}^{\otimes m}\right)\otimes k(0)\ \simeq\ H^0(\wt{X}, \wt{L}^{\otimes m}).$$ As a result, one has that $$V\ =\ \mts{V}_0\ \simeq\ \Proj \bigg(\bigoplus_{m\in \bN} H^0(\wt{X}, \wt{L}^{\otimes m})\bigg)$$ is the ample model of $\wt{L}$ on $\wt{X}$, which implies the normality of $V$ and the birationality of $\phi_0$. 
    
    Consider the restriction $\phi|_{\wt{\mts{Q}}}:\wt{\mts{Q}}\rightarrow \mts{W}$ of $\phi$ to $\wt{\mts{Q}}$. As $\mts{W}$ is an irreducible scheme over $T$, whose general fiber is a line, then $\mts{W}$ is a surface, and $\mts{W}_0$ is also a curve of degree $1$ with respect to $L_V$, hence a line. On the other hand, if $C\subseteq\wt{X}$ is a curve such that $(C.\wt{L})=0$, then $C\subseteq \wt{Q}$ because $2\wt{L}-(1-\epsilon)\wt{Q}$ is ample for $0<\epsilon \ll 1$. Thus the last statement is proved.
    
\end{proof}

\begin{prop}\label{Gorensteincan}
    The variety $V$ is a Gorenstein canonical Fano variety. Moreover, the $\bQ$-Cartier Weil divisor $L_V$ is Cartier on $V$, and $-K_V\sim 2L_V$.
\end{prop}

\begin{proof}
    From the linear equivalence $-K_{\wt{X}}\sim 2\wt{L}-\wt{Q}$ and Lemma \ref{lem:contract-Q}, we know that $$-K_{V}\ =\ (\phi_0)_*(-K_{\wt{X}})\ \sim\ 2L_V$$ is an ample $\mb{Q}$-Cartier divisor. Thus $V$ is a $\bQ$-Fano variety as $X$ is klt and $\phi_0^* K_V = K_{\wt{X}} - \wt{Q}\leq K_{\wt{X}}$. 
    
    Next, we show that $2L_V$ is Cartier, which implies that $V$ is a Gorenstein canonical Fano variety.
    Let $\wt{S}\in |-K_{\wt{X}}|$ be a general member, and $S_V:=(\phi_0)_* \wt{S}$ be its pushforward as a Weil divisor on $V$. By Theorem \ref{generalelephant} we know that $(\wt{X}, \wt{S})$ is a plt log Calabi--Yau pair. Thus $(V,S_V)$ is also a plt log Calabi--Yau pair, which implies that $S_V$ is normal and $\phi_0|_{\wt{S}}:\wt{S}\to S_V$ is birational. As a result,  $(S_V,L_V|_{S_V})$ is a degree $8$ polarized K3 surface as it is the ample model of the quasi-polarized K3 surface  $(\wt{S},\wt{L}|_{\wt{S}})$  of degree $8$. In particular, $L_V|_{S_V}$ is an ample Cartier divisor on $S_V$, and $\big|2L_V|_{S_V}\big|$ is base-point-free by Theorem \ref{Mayer}. By Lemma \ref{isomorphismonsection2}, we have that  $$H^0(V,\mtc{O}_{V}(2L_V))\ \simeq\ H^0(\wt{X},\mtc{O}_{\wt{X}}(2\wt{L}))\ \longrightarrow \ H^0(\wt{S},\mtc{O}_{\wt{S}}(2\wt{L}|_{\wt{S}}))\ \simeq \ H^0(S_V,\mtc{O}_{S_V}(2L_V|_{S_V}))$$ is surjective, and hence $|2L_V|$ is base-point-free. 
    Thus $2L_V$ is Cartier.

    Finally we show that $L_V$ is Cartier.
    By Lemma \ref{lem:contract-Q} we know that $\phi_0$ induces an isomorphism between $\wt{X}\setminus \wt{Q}$ and $V\setminus W$, where $W:=\mts{W}_0$ is the central fiber of $W\rightarrow T$. Since $\wt{L}$ is Cartier on $\wt{X}$, it suffices to show that $L_V$ is Cartier near $W$. As $L_V$ is a $\mb{Q}$-Cartier Weil divisor and $V$ is klt, we know that $\mtc{O}_{V}(L_V)$ is a Cohen-Macaulay divisorial sheaf by \cite[Corollary 5.25]{KM98}. Since $S_V\sim 2L_V$ is Cartier, the restriction $\mtc{O}_{V}(L_V)\otimes_{\mtc{O}_V}\mtc{O}_{S_V}$ is also Cohen-Macaulay. Thus one has a natural isomorphism $$\mtc{O}_{V}(L_V)\otimes_{\mtc{O}_V}\mtc{O}_{S_V}\ \simeq\ (\mtc{O}_{V}(L_V)\otimes_{\mtc{O}_V}\mtc{O}_{S_V})^{**} \ \simeq \ \mtc{O}_{S_V}(L_V|_{S_V}).$$ Notice that $S_V$ contains $W$. As a consequence, for any point $P\in W$, one has $$\dim_k(\mtc{O}_{V}(L_V)\otimes k(P))=\dim_k(\mtc{O}_{S_V}(L_V|_{S_V})\otimes k(P))=1$$ because $L_V|_{S_V}$ is a Cartier divisor on $S_V$. Therefore, the sheaf $\mtc{O}_{V}(L_V)$ is invertible, and hence $L_V$ is Cartier.

\end{proof}

\begin{corollary} \label{cor:V-cubic}
    The linear system $|L_V|$ on $V$ is very ample and embeds $V$ into $\mb{P}^5$ as a complete intersection of two quadric hypersurfaces, which is a Gorenstein canonical quartic del Pezzo threefold. Moreover, the exceptional locus $W$ is a line in $\mb{P}^5$.
\end{corollary}

\begin{proof}
    We proved that $V$ is a Gorenstein canonical Fano variety with $-K_V\sim 2L_V$, where $L_V$ is a Cartier divisor (cf. Proposition \ref{Gorensteincan}). Moreover, we have $(L_V^3) = (\wt{L}^3) = (\wt{\mtc{L}}_t^3) = 4$ for a general $t\in T$.  Hence the statement follows immediately from \cite{Fuj90}.
    
\end{proof}


\begin{lemma}\label{lem:doublept}
    The variety $V$ is generically smooth along $W$. 
\end{lemma}

\begin{proof}
    For a general element $\wt{S}\in|-K_{\wt{X}}|$, the image ${S}_V:=\phi_0(\wt{S})$ is normal, and hence it is the ample model of $\wt{L}|_{\wt{S}}$. In particular, $({S}_V,L_V|_{{S}_V})$ is a degree $8$ polarized K3 surface and hence a $(2,2,2)$-complete intersection in $\mb{P}^5$. Let $\wt{S}_t\in |-K_{\wt{\mts{X}}_t}|$ be deformations of $\wt{S}$. Since the image of $\wt{S}_t$ contains $\mts{W}_t$, then one has that $W\subseteq S_V$. If $W$ is contained the singular locus of $V$, then $S_V$ is double along $W$, which contradicts the fact that ${S}_V$ is normal.

\end{proof}

\begin{prop}\label{prop:blow-up-iso}
    There exist natural isomorphisms $$\Bl_{\mts{W}}\mts{V}\ \simeq\  \wt{\mts{X}}\ \simeq\ \mts{X}$$ over $T$, where $\mts{W}=\phi(\wt{\mts{Q}})$. In particular, we have $\Bl_W V \simeq \wt{X}\simeq X$.
\end{prop}

\begin{proof}
    We first show that the blow-up is compatible with taking fibers. As both $\wt{\mtc{L}}$ and $-K_{\wt{\mts{X}}}$ are nef and big over $T$ and $\wt{\mts{X}}$ is klt, then the Kawamata--Viehweg vanishing theorem implies that $$R^i \wt{\pi}_* \wt{\mtc{L}} \ =\  0$$ for any $i> 0$ and $m\in \bN$. Thus by cohomology and base change, the sheaf $\wt{\pi}_* \wt{\mtc{L}}$ is locally free of rank $6$. By shrinking the base $T$, we may assume that $T=\Spec R$, where $R$ is a DVR with parameter $t$, and $\mts{V}\subseteq \mb{P}^5\times T$. For simplicity, we assume that $\mts{W}\subseteq \mb{P}^5\times T$ is defined by the ideal $(x_2,...,x_5)$, and $\mts{V}$ is defined locally near the generic point of $\mts{W}$ by the ideal $\big(f_1(\underline{x},t),f_2(\underline{x},t)\big)$. For $i=1,2$ we write $$f_i(\underline{x},t)\ =\ x_0l_{i,0}(\underline{x},t)+x_1l_{i,1}(\underline{x},t)+q_i(\underline{x},t)$$ where $\underline{x}=(x_2,...,x_4)$, and $l_{i,j}(\underline{x},t)$ (resp. $q_i(\underline{x},t)$) is a linear (resp. quadratic) polynomial in variables $\underline{x}$ with coefficients in $R$. Moreover, by Lemma \ref{lem:doublept} we know that the linear terms (with respect to $\underline{x}$) of $f_1(\underline{x},t)$ and $f_2(\underline{x},t)$ are not proportional for any $t$. In order to prove that $\Bl_{\mts{W}_0}\mts{V}_0\simeq (\Bl_\mts{W}\mts{V})|_{t=0}$, it suffices to show that the natural morphism $$\frac{(x_2,x_3,x_4,x_5)^k}{t(x_2,...,x_5)^k+(f_1,f_2)} \ \twoheadrightarrow\ \left(\frac{(x_2,x_3,x_4,x_5,t)}{(t,f_1,f_2)}\right)^k$$ is an isomorphism for any $k\gg0$. Indeed, for any polynomial $g(\underline{x},t)$ such that the degree of each monomial with respect to $\underline{x}$ is at least $k$, if $g(\underline{x},0)$ is contained in the ideal generated by $f_i(\underline{x},0)$ for $i=1,2$, then $$g(\underline{x},0)\ =\ f_1(\underline{x},0)\cdot h_1(\underline{x})+f_2(\underline{x},0)\cdot h_2(\underline{x})$$ for some polynomial $h$ of degree at least $k-1$. Consequently, one has $$g(\underline{x},t)-f_1(\underline{x},t)\cdot h_1(\underline{x})-f_2(\underline{x},t)\cdot h_2(\underline{x})\ =\ t\cdot p(\underline{x},t)$$ such that the $\underline{x}$-degree of each monomial in $l(\underline{x},t)$ is at least $k$ as desired.

    As an immediate result, we see that $\Bl_{\mts{W}}\mts{V}$ is normal. In fact, since $\Bl_WV=\Bl_{\mts{W}_0}\mts{V}_0$ is an integral scheme, then it satisfies (R$_0$) and (S$_1$) conditions. Thus $\Bl_{\mts{W}}\mts{V}$ satisfies (R$_1$) and (S$_2$) conditions, which is equivalent to being normal.

    Notice that we have the desired isomorphism over $T^{\circ}$. Thus $\Bl_{\mts{W}}\mts{V}$ and $\wt{\mts{X}}$ are birational and isomorphic in codimension one. Let $\mts{Q}'$ be the exceptional divisor of the blow-up $\psi:\Bl_{\mts{W}}\mts{V}\rightarrow \mts{V}$, and $\mtc{L}'=\psi^{*}\phi_* \wt{\mtc{L}}\sim_T \psi^* \mtc{O}_{\mts{V}}(1)$. From the blow-up construction we know that the $\mb{Q}$-Cartier $\bQ$-divisor $\mtc{L}'-\epsilon\mts{Q}'$ is ample over $T$ for any rational number $0<\epsilon\ll1$. On the other hand, since $\wt{\mtc{L}}$ is nef over $T$ (cf. Proposition \ref{nefness}), and $2\wt{\mtc{L}}-(1-\epsilon)\wt{\mts{Q}}$ is ample over $T$ for any rational number $0<\epsilon\ll 1$ (from the construction of $\wt{\mts{X}}$), we know that  $\wt{\mtc{L}}-\epsilon\wt{\mts{Q}}$ is also ample over $T$. 
    Since $\Bl_{\mts{W}}\mts{V}$ and $\wt{\mts{X}}$  are isomorphic in codimension $1$, and $\wt{\mtc{L}}-\epsilon\wt{\mts{Q}}$ is the birational transform of $\mtc{L}'-\epsilon \mts{Q}'$, we conclude that $\Bl_{\mts{W}}\mts{V}\simeq  \wt{\mts{X}}$. As a result, we see that $-K_{\wt{\mts{X}}}$ is ample over $T$, and hence $\wt{\mts{X}}\simeq\mts{X}$.
    
\end{proof}

\begin{corollary}\label{cor:bpf}
    The complete linear series $|-K_{X}|$ is base-point-free. 
\end{corollary}

\begin{remark}\label{rem:cubic-blowup-fano}
    \textup{Let $\ell$ be a line in $\mb{P}^5$, and $\pi:\Bl_{\ell}\mb{P}^5\rightarrow \mb{P}^5$ be the blow-up. Set $L:=\mtc{O}_{\mb{P}^5}(1)$ to be the pull-back of the class of hyperplane section and $E$ to be the class of exceptional divisor. If $V$ is a $(2,2)$-complete intersection in $\mb{P}^5$ containing a line $\ell$ such that $V$ is generically smooth along $\ell$, then the blow-up $X$ of $V$ along $\ell$ admits a natural ample line bundle $(2L-E)|_{X}$, which coincides with the anti-canonical line bundle of $X$ by adjunction. If moreover $V$ has Gorenstein canonical singularities and the projective tangent cone of $V$ at $P$ is a normal quadric surface, then by inversion of adjunction, the polarized variety $(X,(2L-E)|_{X})$ is an anti-canonically polarized Gorenstein canonical Fano variety.}
\end{remark}

\begin{proof}[Proof of Theorem \ref{23complete}]
    It follows from Theorem \ref{thm:nonvanishing}(1) that $X$ is Gorenstein canonical. By Corollary \ref{cor:V-cubic}, Lemma \ref{lem:doublept} and Proposition \ref{prop:blow-up-iso}, we have $X\simeq \Bl_{\ell}V$ where $V$ is a Gorenstein canonical quartic del Pezzo threefold and $\ell\subseteq V$ is a line. Moreover, $X$ is isomorphic to the blow-up of $\bP^3$ along some curve by Lemma \ref{lem:sarkisov link}.
\end{proof}

\section{Variation of GIT for lines in quartic del Pezzo threefolds}\label{Sec:5}

Let $\bfV$ be a $6$-dimensional vector space, $\bP^5:=\bP\bfV$ be its projectivization, and $(\bP^5)^*:=\bP\bfV^*$ be its dual space. Then a line $\ell$ on $\bP^5$ can be identified as a 3-plane in $(\bP^5)^*$, which is further identified with a $(1,1)$-complete intersection in $(\bP^5)^*$. Let $$\bG(1,5) \ :=\ \Gr(2,\bfV^*)  \ = \ \Gr(\bfV,2)$$ be the Grassmannian parametrizing lines in $\mb{P}^5$, and $$\bG(1,20) \ := \ \Gr(2,\Sym^2\bfV)$$ be the Grassmannian parametrizing pencils of quadric hypersurfaces in $\mb{P}^5$. For a pencil $\mtc{P}$ of quadrics, let $V$ be the base locus of $\mtc{P}$, which is a $(2,2)$-complete intersection if $\mtc{P}$ is general. Let $W\subseteq \bG(1,5)\times \bG(1,20)$ be the incidence variety \begin{equation}\label{eq:defn W}
     W\ := \ \big\{(\ell,V)\ | \ \ell \subseteq V\big\},
 \end{equation} and $U\subseteq W$ be the big open subset of $W$ consisting of pairs $(\ell,V)$ such that $V$ is a $(2,2)$-complete intersection with canonical singularities and $\ell$ is not contained in the singular locus of $V$. Notice that $W$ is a $\bG(1,17)$-bundle over $\bG(1,5)$, hence the Picard group of $W$ is free of rank two, generated by $$\xi:=p_1^{*}\mtc{O}_{\bG(1,5)}(1),\ \ \ \textup{and}\ \ \  \eta:=p_2^{*}\mtc{O}_{\bG(1,20)}(1).$$ Moreover, since $U$ is a big open subset of $W$, then we can identify $\Pic(U)$ with $\Pic(W)$.

\subsection{VGIT of lines in quartic del Pezzo threefolds}

Let $x_0,...,x_5$ be a basis of $\bfV$ so that $\bP^5$ has homogeneous coordinate $[x_0,...,x_5]$, and $a_0,...,a_5$ the dual basis of $\bfV^*$. Let $\lambda$ be a diagonal one-parameter subgroup (abbv. 1-PS) of $\SL(\bfV)$ such that the weight of $(x_0,...,x_5)$ is $(r_0,...,r_5)$, so that it induces a diagonal action of $\bfV^*$ of weight $(-r_0,...,-r_5)$.

Let \(p_{ij}:=a_i\wedge a_j\) be the Pl\"{u}cker coordinates on the Grassmannian \(\bG(1,5)\). 
For a one-parameter subgroup \(\lambda\) acting diagonally with weights \(r_0,\ldots,r_5\), the Hilbert--Mumford weight of \(\mtc{O}(1)\) at \([\ell]\) is
\[
    \mu^{\mtc{O}(1)}\big([\ell],\lambda\big)
    \;=\;
    \max\big\{-r_i-r_j \;\big|\; i<j \text{ and } p_{ij}(\ell)\neq 0\big\}.
\] As an example, if the line \(\ell\subseteq \bP^5\) is given by 
\[
    \ell=\bV(x_0,x_1,x_2,x_3),
\]
then it corresponds to the codimension-\(2\) linear subspace \(\bV(a_4,a_5)\subseteq (\bP^{5})^{*}\).  
Its only nonzero Pl\"{u}cker coordinate is \(p_{45}\), i.e.
\[
    p_{45}(\ell)=1,\qquad 
    p_{ij}(\ell)=0 \ \text{ for } \ (i,j)\neq (4,5),
\] and the Hilbert--Mumford weight of \(\mtc{O}(1)\) at \([\ell]\) is $-r_4-r_5$.

If $\mtc{P}\in \bG(1,20)$ is a pencil of quadrics whose base locus $V$ is $3$-dimensional, then $V$ is a quartic del Pezzo threefold, and one can identify $\mtc{P}$ with its base locus $V$. Let $\{x_{ij}:=x_ix_j\}_{0\leq i\leq j\leq 5}$ be the induced basis of $H^0(\bP^5,\mtc{O}_{\bP^5}(2)) = \Sym^2\bfV$. For any two pairs $(i,j)\neq (k,l)$ with $i\leq j$ and $k\leq l$, let $q_{i,j;k,l}:=x_{ij}\wedge x_{kl}$ be the Pl\"{u}cker coordinates of the Grassmannian $\bG(1,20)$. Then the weight of $x_{ij}$ under $\lambda$ is $r_i+r_j$, and it follows that 
\begin{equation}
    \mu^{\mtc{O}(1)}\big(\mtc{P},\lambda\big)\ =\ \max\big\{r_i+r_j+r_k+r_l\mid i<j \textup{ such that }q_{i,j;k,l}(\mtc{P})\neq0 \big\}.
\end{equation}

\begin{defn}
    Let $t\geq 0$ be a rational number. A pair $(\ell,\mtc{P})\in W$ is said to be \emph{$t$-GIT (semi/poly/un)stable} if it is a GIT (semi/poly/un)stable point in $W$ under the action of $\SL(6)$, or equivalently under the action of $\PGL(6)$, with respect to the polarization $\eta+t\xi$.
\end{defn}

\begin{lemma}[Hilbert--Mumford criterion]
    If there exists a diagonal 1-PS $\lambda$ of $\SL(6)$ such that the Hilbert--Mumford index $$\mu^t(\ell,\mtc{P};\lambda)\ :=\ \mu^{\mtc{O}(1)}(\mtc{P},\lambda)+t\mu^{\mtc{O}(1)}([\ell],\lambda)$$ is negative, then $(\ell,\mtc{P})$ is $t$-GIT unstable.
\end{lemma}

We refer the reader to \cite[Section~\S2]{HZ25} for a detailed treatment of the Hilbert--Mumford index of linear systems of hypersurfaces in projective spaces.

\begin{prop}[GIT stability for complete intersections] An intersection $V=Q\cap Q'$ of two quadrics in $\bP^5$ is
\begin{enumerate}
    \item stable if and only if the discriminant of their pencil has no multiple roots, or equivalently, $X$ is smooth;
    \item semistable if and only if the discriminant of their pencil has no roots of multiplicity $>3$;
    \item polystable if and only if the two quadrics can be simultaneously diagonalized, the discriminant of their pencil admits no roots of multiplicity $>3$ and if there is a root of multiplicity exactly equal to $3$ then $V$ is actually isomorphic to $\bV\big(x_0^2+x_1^2+x_2^2,x_3^2+x_4^2+x_5^2\big)$.
\end{enumerate}
\end{prop}

\begin{proof}
    See \cite[Theorem 4.2]{AL00} and its proof.
\end{proof}

Suppose that $\mtc{P}$ is a pencil of quadrics whose base locus $V:=\Bs(\mtc{P})$ is a complete intersection, and $\ell\subseteq \bP^5$ is a line contained in $V$. Up to a change of coordinates, one may assume that the line and two generators of $\mtc{P}$ are of the form \begin{equation}\label{eq:standard form}
   \ell=\bV(x_2,...,x_5),\ \ \ Q=\bV(x_0\ell_0+x_1\ell_1+q), \ \ \ \textup{and} \ \ \ Q'=\bV(x_0\ell_0'+x_1\ell_1'+q'), 
\end{equation} where $\ell_0,\ell_1,\ell_0',\ell_1'$ are linear forms in $x_2,..,x_5$ and $q,q'$ are quadratic forms in $x_2,..,x_5$.

\begin{lemma}\label{lem: no semistable for t large}
    If $t>\frac{1}{2}$, then there are no GIT semistable points in $W$.
\end{lemma}

\begin{proof}
Let $(\mtc{P},\ell)$ be a point in $W$ which is of the form in (\ref{eq:standard form}). Consider the 1-PS $\lambda$ of $\SL(6)$ of weight $(2,2,-1,-1,-1,-1)$. The condition $\mu^t(\mtc{P},\ell;\lambda)=2-4t\geq0$ imposes that $t\leq \frac{1}{2}$.
\end{proof}

Let us analyze the equations in (\ref{eq:standard form}). If $V$ acquires a singularity $p$ on $\ell$, then up to a change of coordinates, one can assume that $p=[1,0,...,0]$. It follows that $\ell_0$ and $\ell_0'$ are linearly dependent, and hence one may assume that $\ell_0'=0$ by replacing $Q'$ by another quadric in the pencil. In particular, under the correspondence in Lemma \ref{lem:sarkisov link}, this comes down to saying that the (unique) quadric containing the curve $C\subseteq \bP^3$ is non-normal: indeed, the quadric is $\bV(\ell_0\ell'_1)$.

\begin{lemma}\label{lem:no singular point on l}
    If $\mtc{P}\in\bG(1,20)$ is a pencil whose base locus $V$ is a complete intersection, and $\ell\subseteq V$ is a line such that $V$ has a singularity $p$ on $\ell$. Then $(\mtc{P},\ell)$ is $t$-GIT unstable for any $t>0$.
\end{lemma}

\begin{proof}
    We may assume that $\ell=\bV(x_2,x_3,x_4,x_5)$, $V$ is generated by two quadrics $Q=\bV(x_0x_2+x_1x_3+q)$ and $Q'=\bV(x_1x_4+q')$ and that $p=[1,0,...,0]$ is a singularity of $V$ on $\ell$. Consider the diagonal 1-PS $\lambda$ of $\SL(6)$ of weight $(1,0,-1,0,0,0)$, with respect to which one has $$\mu^t(\mtc{P},\ell;\lambda)=-t<0$$ for any $t>0$, and thus $(\mtc{P},\ell)$ is $t$-GIT unstable.
\end{proof}

\begin{lemma}
    If $[\mtc{P}]\in \bG(1,20)$ is not a complete intersection, then $(\mtc{P},\ell)$ is $t$-GIT unstable for any $t\geq0$.
\end{lemma}

\begin{proof}
    Let $Q$ and $Q'$ be two quadrics generating $\mtc{P}$. Then one must have that $Q$ and $Q'$ are both reducible and share a common hyperplane. One may thus assume that $$Q=\bV\big(x_0\ell(\underline{x})\big)\ \ \ \textup{and} \ \ Q'=\bV\big(x_0\ell'(\underline{x})\big).$$ Let $\lambda$ be the 1-PS of $\SL(6)$ of weight $(-5,1,1,1,1,1)$. If $(\mtc{P},\ell)$ is $t$-GIT semistable, then one has that $\mu^t(\mtc{P},\ell;\lambda)=-8+4t\geq0$, i.e. $t\geq 2$, which contradicts Lemma \ref{lem: no semistable for t large}.
\end{proof}

\begin{prop}\label{prop:no wall before 1/8}
    Suppose that $V$ is a GIT unstable $(2,2)$-complete intersection in $\bP^5$. Then there exists a 1-PS $\lambda(t)$ such that the limit $V_0:=\lim_{t\to 0}\lambda(t)\cdot V$ is also a complete intersection, and the $t$-GIT weight of the pair $(V,\ell)$ under $\lambda$ is negative for any line $\ell\subseteq V$ and any $t<\frac{5}{26}$.
\end{prop}

\begin{proof}
    In \cite[Proposition A.2]{SS17}, the authors explicitly construct a 1-PS $\lambda$ such that the corresponding limit $V_0$ has negative GIT weight and is a complete intersection. We claim that the $t$-GIT weight of the pair $(V,\ell)$ under $\lambda$ is also negative, for any $\ell$ and any $t<\frac{5}{26}$.

A classical result of Weierstrass and Segre (cf. \cite{AL00}) allows one to put a pencil $\mtc{P}$ of quadrics into a canonical form $\mtc{P} = uQ + vQ'$, where $Q$ and $Q'$ are block diagonal. A systematic treatment can be found in \cite[pp.~286–292]{HP94}. Let $r_0$ denote the number of variables that do not appear in $\mtc{P}$. The discriminant $\Delta(\mtc{P})$ defines the subscheme of $\bP^1$ parameterizing singular quadrics in the pencil. In \cite[pp.289, Eqns.(1)–(2)]{HP94}, two matrices $A$ and $B$ are introduced which play a role in the classification. For our purposes, since our quadrics are $4$-dimensional, there are only a few cases to consider, so we do not define them here; instead, we write down the explicit equations of the quadrics as needed.

    \textbf{Non-degenerate case}
        This is the case where $\Delta(\mtc{P})$ is not identically zero. If $\mtc{P}$ is non-degenerate and GIT unstable, then $\Delta(\mtc{P})$ has a root of multiplicity $\geq 4$. A general pencil satisfying this has Segre symbol
        \begin{enumerate}
            \item either $[4,1,1]$, i.e. it is defined by the equations $$\begin{cases}
                x_0x_3+x_1x_2+x_4^2+x_5^2 =0 \\
                x_1x_3+x^2_2+ax_4^2+bx_5^2 =0 
            \end{cases},$$ where $a,b\neq1$ are two different numbers.
            \item or $[4,2]$, i.e. it is defined by the equations $$\begin{cases}
                x_0x_3+x_1x_2+x_5^2+2ax_4x_5 =0 \\
                x_1x_3+x^2_2+2x_4x_5 =0 
            \end{cases},$$ where $a\neq1$ is a complex number.
        \end{enumerate}  In both cases, one can take the diagonal 1-PS $\lambda$ of weight $(5,2,-1,-4,-1,-1)$. With respect to this weight, one has that $$\mu(\mtc{P},\lambda)=\max\{1,-2\}-2=-1, \ \ \textup{and} \ \ \mu([\ell],\lambda)\leq -(-4-1)=5,$$ and thus $$\mu^{t}(\mtc{P},\ell;\lambda)\ \leq \ -1+ 5t\ <\ 0$$ for any line $\ell\subseteq \Bs(\mtc{P})$ and any rational number $0\leq t<\frac{5}{26}$.

        \textbf{Pure degenerate case $r_0=0$ and $A=B=0$} There is only one case to deal with, i.e. when $Q=\bV(x_0x_1+x_2x_3)$ and $Q'=\bV(x_1x_2+x_3x_4)$. One can take the 1-PS $\lambda$ of weight $(2,-2,1,-1,0,0)$, with respect to which one has $$\mu^{t}(\mtc{P},\ell;\lambda)\ \leq \ -1+\frac{1}{5}\cdot 3\ <\ 0.$$

    \textbf{Mixed general case $r_0 = 0$, but $A,B \neq 0$} There is only one case to deal with, i.e. when $Q=\bV(x_0x_1+x_2x_3+x_5^2)$ and $Q'=\bV(x_1x_2+x_3x_4)$. One can take the 1-PS $\lambda$ of weight $(1,-3,3,-5,5,-1)$, with respect to which one has $$\mu^{t}(\mtc{P},\ell;\lambda)\ <\  -2+\frac{1}{5}\cdot 8\ <\ 0.$$

    \textbf{Case $r_0\neq0$ in the canonical form} This is an easy case: $\mtc{P}$ has non-finite isotropy group and one can explicitly write down a destabilizing 1-PS that fixes $\mtc{P}$ with negative GIT weight, because one has the freedom to choose weight for the last $r_0$ variables. For example, suppose that no \(x_0\) appears in the defining equations of the quadrics in the pencil \(\mtc{P}\). 
Then one may take the \(1\)-PS with weights \((5m,-m,-m,-m,-m,-m)\) for some \(m>0\).  
In this case, the Hilbert--Mumford weight satisfies
\[
    \mu^{t}(\mtc{P},\ell;\lambda)\ \le\ -4m+2mt \,<\, 0
    \qquad\text{for any } t<1.
\]
\end{proof}

\section{K-moduli space of family №2.19}

In this section, we will show that the K-moduli of family №2.19 is isomorphic to a VGIT moduli, and prove the main theorems displayed in the introduction.

\subsection{Computation of CM line bundles}

Let $(\mts{L},\mts{V})\rightarrow U$ be the universal family over $U$, and $\sX:= \Bl_{\sL}\sV$ be the blow-up. We will compute the CM line bundle $\Lambda_{\CM,f}$ associated to the family $f:\sX\rightarrow U$ of Fano threefolds. One can write \begin{equation}\label{eq:CM}
    \Lambda_{\CM,f}\ = \ a\xi+b\eta.
\end{equation}

\subsubsection{Testing curve I}

Fix a line $L_0\subseteq \bP^5$, and a smooth quadric hypersurface $Q$. Let $\{Q_t\}_{t\in\mb{P}^1}$ be a general pencil of quadrics containing $L_0$. Let $\mts{V}_1\subseteq \mb{P}^5\times\mb{P}^1$ be the total space of the family of complete intersections $\{Q_t\cap Q\}_{t\in\mb{P}^1}$, which is a complete intersection by divisors of classes $\mtc{O}_{\mb{P}^5\times\mb{P}^1}(2,0)$ and $\mtc{O}_{\mb{P}^5\times\mb{P}^1}(2,1)$; and $$\mts{L}\ :=\ L_0\times\mb{P}^1\ \subseteq \ \mts{V}_1\ \subseteq\ \mb{P}^5\times \mb{P}^1$$ be the (trivial) family of lines. Then $$f_1:\ \mts{X}_1:=\Bl_{\mts{L}}\mts{V}_1\ \longrightarrow\ C_1\ \simeq \ \mb{P}^1$$ is a one parameter family of Fano threefolds. Then we have $$(C_1.\xi)=0,\ \ \textup{and}\ \ \ (C_1.\eta)=1.$$
It follows from the short exact sequence $$0\ \longrightarrow \ N_{\mts{L}/\mts{V}_1}\ \longrightarrow \ N_{\mts{L}/\mb{P}^5\times\mb{P}^1}\ \longrightarrow \ N_{\mts{V}_1/\mb{P}^5\times\mb{P}^1}|_{\mts{L}}\ \longrightarrow \ 0$$ that $$c_1(N_{\mts{L}/\mb{P}^5\times\mb{P}^1})\ =\ \mtc{O}_{\bP^1\times\bP^1}(4,0), \ \ \ c_1(N_{\mts{V}_1/\mb{P}^5\times\mb{P}^1}|_{\mts{L}})\ =\ \mtc{O}_{\bP^1\times\bP^1}(4,1),$$ $$c_2(N_{\mts{L}/\mb{P}^5\times\mb{P}^1})\ =\ 0, \ \ \textup{and} \ \ c_2(N_{\mts{V}_1/\mb{P}^5\times\mb{P}^1}|_{\mts{L}})\ =\ 2,$$ and hence $$c_1(N_{\mts{L}/\mts{V}_1})\ =\ \mtc{O}_{\bP^1\times \bP^1}(0,-1),\ \ \ \textup{and}\ \ \ c_2(N_{\mts{L}/\mts{V}_1})\ =\ 2.$$
Let $\mts{E}_1\subseteq \mts{X}_1$ be the exceptional divisor of the blow-up $\pi_1:\mts{X}_1\rightarrow \mts{V}_1$. Then by intersection formula for blow-ups (cf. \cite[Chapter 13]{EH16}), one has that $$(\mts{E}_1^4)\ =\ -c_1(N_{\mts{L}/\mts{V}_1})^2+c_2(N_{\mts{L}/\mts{V}_1})\ = \ 2,$$ 
$$(\mts{E}_1^3.\pi^{*}K_{\mts{V}_1})\ =\ -\big(K_{\mts{V}_1}|_{\mts{L}}.\ c_1(N_{\mts{L}/\mts{V}_1})\big)\ =\ -2,\ \ \textup{and}$$
$$\big(\mts{E}_1^2.(\pi^{*}K_{\mts{V}_1})^2\big)\ =\ -\big(K_{\mts{V}_1}|_{\mts{L}}\big)^2\ =\ -4.$$ It follows that 
\begin{equation}
    \begin{split}
        \deg \Lambda_{\CM,f_1}&\ =\ -(K_{\mts{X}_1/\bP^1})^4\\ & \ =\ -(K_{\mts{X}_1})^4+8(-K_\textup{№2.19})^3\\
        &\ =\ -(\pi^{*}K_{\mts{V}_1}+\mts{E}_1)^4+8\cdot 26\\
        &\ = \ -(2) - 4\cdot(-2) - 6\cdot (-4)-160 + 208\\
        &\ = \ 78.
    \end{split}
\end{equation}
In particular, one sees that the coefficient $b$ in the Equation (\ref{eq:CM}) is $78$.

\subsubsection{Testing curve II}

Let $V$ be a general $(2,2)$-complete intersection in $\bP^5$, and $A:=F_1(V)\subseteq \bG(1,5)\subseteq \bP^{14}$ be the Fano scheme of lines on $V$. One can identify $V$ with the moduli space of stable vector bundles on a smooth genus $2$ curve $C$ of rank $2$ and of a fixed determinant $\xi$ of degree $3$. The Fano scheme of lines $A$ is isomorphic to $\Pic^0(C)$. Let $\Theta$ be the Theta-divisor on $A$, which is isomorphic to $C$. We first claim that the restriction of $\mtc{O}_{\bP^{14}}(1)$ on $A$ is linearly equivalent to $4\Theta$. As $A$ is a principally polarized abelian surface with $(\Theta^2)=2$ and $\mtc{O}_{\bP^{14}}(1)|_{\bG(1,5)}=\xi$, it suffices to check that $(\xi^2.A)=32$. We use $\sigma_{i_1,...,i_k}$ to denote the Schubert classes on $\bG(1,5)$. Then $\xi^2=\sigma_1^2=\sigma_2+\sigma_{1,1}$. Since a general hyperplane section of $V\subseteq \bP^5$ is a smooth del Pezzo surface of degree $4$, which contains 16 lines, then $(\sigma_{1,1}.A)=16$; since a general 2-plane section of $V$ consists of 4 points and through each of them there are 4 lines, then $(\sigma_{2}.A)=16$ as desired.

Let $i:B\subseteq V\times A $ be the total space of lines, and $\pi:\mts{X}:=\Bl_{B}(V\times A)\rightarrow V\times A$ be the blow-up. Then the composition $\phi:\mts{X}\rightarrow A$ is a family of smooth №2.19. One has that $$-K_{\wt{\mts{X}}}\ = \pi^{*}(2H)-Z,$$ where $Z$ is the $\pi$-exceptional divisor, $H=p_1^{*}\mtc{O}_{\bP^5}(1)|_{V}$, and $p_1:V\times A\rightarrow V$ and $p_2:V\times A\rightarrow A$ are the two projections. Notice that the composition $\alpha:=p_2 \circ i:B\rightarrow A$ is a $\bP^1$-bundle over $A$ associated to the restriction of the universal subbundle $\mtc{S}$ of $\bG(1,5)$. Here, our projective bundles and Grassmannians parametrize subspaces instead of quotient spaces. We summarize by the following commutative diagram:

$$\xymatrix{
  \mts{X}:=\Bl_{B}(V\times A) \ar[drrrr]_{\phi} \ar[rr]^{\pi}  & & V\times A \ar[rr]^{p_1} \ar[drr]^{p_2}   &  &  V\subseteq \bP^5 \\
  Z \ar[rr]^{\beta} \ar@{^(->}[u]^{j} & & B \ar[rr]^{\alpha} \ar@{^(->}[u]^{i} & & \ A\subseteq \bG(1,5).
 }$$
We restrict the family $\phi:\mts{X}\rightarrow A$ to the curve $\Theta$, and the degree of CM line bundle $\deg\lambda_{\CM,\Theta}$ is equal to 
\begin{equation}
\begin{split}
    \deg\lambda_{\CM,\Theta}\ & =\ -\big( \phi_{*}(-K_{\mts{X}})^4 \ . \ \Theta \big)\\
    & = \ -\big( \pi_{*}(-K_{\mts{X}})^4 \ . \ p_2^{*}\Theta\big) \\
    & = \ - \big(24(H^2.\pi_*(Z^2).p_2^*\Theta) + 8(H.\pi_*(-Z^3).p_2^*\Theta)+ (\pi_*(Z^4).p_2^*\Theta)  \big).
\end{split}
\end{equation}
We first need to compute the Chern classes of the tangent bundle $T_{B}$. 
By the exact sequence $$0\ \longrightarrow \ T_{B/A} \ \longrightarrow \ T_{B}\ \longrightarrow \ \alpha^*{T_A}\simeq \mtc{O}_{B}^{\oplus2} \ \longrightarrow \ 0$$ and the fact $B=\bP(\mtc{S}|_A)$, one has that  $$c_1(T_B)\ =\ c_1(T_{B/A})\ =\ -K_{B/A}\ =\ 2H|_B-\alpha^{*}c_1(\mtc{S})|_{A}\ = \ 2H|_{B}-\alpha^{*}(4\Theta) ,$$ and that $c_2(T_B)=c_2(T_{B/A})=0$ since $T_{B/A}$ is a line bundle.
We now compute the Chern classes of the normal bundle $N_{B/V\times A}$ from the defining sequence $$0\ \longrightarrow \ 
T_B \ \longrightarrow \ T_{V\times A}|_B\ \longrightarrow\  N_{B/V\times A}\ \longrightarrow 0.$$ Since $A$ is an abelian surface and $V$ is a $(2,2)$-complete intersection, then one has $$c(T_{V\times A}|_B)\ =\ c(T_{V})|_B\ = \frac{(1+H_V)^6}{(1+2H_V)^2}\bigg|_B\ =\ 1+2H|_B+3H|_B^2+\cdots,$$ where $H|_B$ coincides with $\mtc{O}_{\bP(\mtc{S}|_A)}(1)$ under the identification $B=\bP(\mtc{S}|_A)$. Thus we have $$c_1(N_{B/V\times A})\ =\ \alpha^*(4\Theta),\ \ \ \textup{and}\ \ \ c_2(N_{B/V\times A})\ = \ 3H|_B^2+16\alpha^*\Theta^2-8H|_B.\alpha^*\Theta.$$ By the Chow ring of projective bundles, one has $$H|_B^2-4H|_B.\alpha^{*}\Theta+\alpha^*c_2(\mtc{S}|_A)\ = \ 0.$$ Moreover, we have that $Z\simeq\bP(N_{B/V\times A})$, and we denote $\zeta=\mtc{O}_{\bP(N_{B/V\times A})}(1)$, which satisfies $$\zeta^2+\beta^*c_1(N_{B/V\times A})\zeta+\beta^*c_2(N_{B/V\times A})\ =\ 0.$$ Then we have the following relations
$$(H^2.\pi_{*}Z^2.p_2^*\Theta)\ =\ -(H^2.\pi_{*}j_{*}\zeta.p_2^*\Theta)  \ = \ -\big((H|_{B})^2.\alpha^{*}\Theta\big)_{B}\ =\ (-4\Theta^2)_A\ =\ -8,$$

$$(H.-\pi_{*}Z^3.\Theta)\ =\ -\big(H.\pi_{*}j_{*}\zeta^2.p_2^{*}\Theta\big)\ =\ -\big((H|_{B}).-c_1(N_{B/V\times A}).\alpha^*\Theta\big)_{B}\ =\ (4\Theta^2)_A \ =\ 8,$$

$$(\pi_{*}Z^4.\Theta)\ =\ -\big(\pi_*j_*\zeta^3.p_2^{*}\Theta\big)\ =\  -\big(c_1(N_{B/V\times A})^2-c_2(N_{B/V\times A}).\alpha^*\Theta\big)_B \ = \ (4\Theta^2)_A\ =\ 8.$$
It follows that $$\deg\lambda_{\CM,\Theta}\ =\ -(24\cdot (-8)+8\cdot 8+8)\ = \ 120,$$ and hence  $$8a\ =\ a\cdot (\Theta.\xi)\ =\ (\Theta.a\xi+b\eta) \ =\ \deg \lambda_{\CM,\Theta} \ =\ 120,$$ which implies $a=15$.

To summarize, one has the following result.

\begin{prop}
    The CM line bundle $\Lambda_{\CM,f}$ associated to the family $f:\mts{X}\rightarrow U$ is proportional to the class $5\xi+26\eta$.
\end{prop}

\subsection{K-moduli of №2.19 and VGIT quotient}

\begin{prop}\label{prop:K implies GIT}
    Let $X=\Bl_{\ell}V$ be a blow-up of a $(2,2)$-complete intersection $V$ in $\bP^5$ along a line. 
    \begin{enumerate}
        \item If $X$ is K-semistable, then $V$ is a K-semistable quartic del Pezzo threefold.
        \item If $X$ is K-(semi/poly)stable, then $(V,\ell)$ is $t_0$-GIT (semi/poly)stable, where $t_0=\frac{5}{26}$.
    \end{enumerate} 
    In particular, if $X$ is K-semistable, then $X$ is isomorphic to the blow-up of a K-semistable $(2,2)$-complete intersection along a smooth line.
\end{prop}

\begin{proof}
    If $V$ is K-unstable, then by \cite[Proposition A.2]{SS17}, one can find an explicit diagonal 1-PS $\lambda$ of $\SL(6)$ such that the limit $V_0:=\lim_{t\to 0}\lambda(t)\cdot V$ is also a complete intersection of two quadrics, and $\mu(\mtc{P};\lambda)<0$. In fact, using the same 1-PS exhibited in \cite[Proposition A.2]{SS17}, one can show that $\mu^{t_0}(\mtc{P},\ell';\lambda)<0$ for any line $\ell'\subseteq V_0$; see Proposition \ref{prop:no wall before 1/8}. Let $\mts{V}\rightarrow \bA^1$ be the test configuration of $V$ induced by $\lambda$, let $\mts{L}$ be the closure of $\ell\times (\bA^1\setminus\{0\})$ in $\mts{V}$, and $\ell_0:=\mts{L}\cap V_0$ be the limit of $\ell$ under this degeneration, which is also a line. Then $\mts{X}:=\Bl_{\mts{L}}\mts{V}\rightarrow \bA^1$ is a test configuration of $X=\Bl_{\ell}V$. If $X$ is K-semistable, then one has $$ \mu^{t_0}(V,\ell;\lambda) \ =\ c\cdot \mu^{\lambda_{\CM}}(\mts{X}_0;\lambda) \ = \ c'\cdot \Fut(\mts{X}) \ \geq\ 0,$$ where $c,c'>0$ are two constants. This leads to a contradiction, and hence (1) holds true. This argument also shows that there is no GIT walls in $(0,t_0]$.

    Suppose now $X$ is K-semistable but $(V,\ell)$ is $t_0$-GIT unstable. Then by the structure of variation of GIT, there exists a 1-PS $\lambda$ of $\SL(6)$ such that $\mu^0(V,\ell;\lambda)=0$ and $\mu^t(V,\ell;\lambda)<0$ for any $t>0$. Since $V$ is GIT semistable by (1), then the limit $V_0:=\lim_{s\to0}\lambda(s)\cdot V$ is also GIT semistable, and hence is a $(2,2)$-complete intersection. Let $\mts{V}\rightarrow \bA^1$ be the test configuration of $V$ induced by $\lambda$, let $\mts{L}$ be the closure of $\ell\times (\bA^1\setminus\{0\})$ in $\mts{V}$, and $\ell_0:=\mts{L}\cap V_0$ be the limit of $\ell$ under this degeneration, which is also a line. Then $\mts{X}:=\Bl_{\mts{L}}\mts{V}\rightarrow \bA^1$ is a test configuration of $X=\Bl_{\ell}V$. Since $X$ is K-semistable, then one has $$0\ >\ \mu^{t_0}(V,\ell;\lambda) \ =\ c\cdot \mu^{\lambda_{\CM}}(\mts{X}_0;\lambda) \ = \ c'\cdot \Fut(\mts{X}) \ \geq\ 0,$$ which is a contradiction. The K-(poly)stability part in statement (2) is similar.

    The last statement follows from Lemma \ref{lem:no singular point on l}.
\end{proof}

\begin{corollary}
    Let $X$ be a K-semistable Fano variety in $\mts{M}^{K}_{\textup{\textnumero 2.19}}$. Under the correspondence in Lemma \ref{lem:sarkisov link}, $X$ is the blow-up of $\bP^3$ along a curve $C$. Then the unique quadric surface containing $C$ is normal.
\end{corollary}

\begin{proof}
    This follows immediately from Proposition \ref{prop:K implies GIT}, Lemma \ref{lem:no singular point on l} and the paragraph above it.
\end{proof}

\begin{lemma}[Fano scheme of lines]\label{lem: fano scheme of lines}
    Let $X$ be the blow-up of a $(2,2)$-complete intersection $V$ in $\bP^5$ along a line $\ell$ which is not contained in the singular locus of $X$. Then the Fano scheme of $(-K_X)$-lines consists of two connected components. Furthermore, only one of these two components is a smooth rational curve.
\end{lemma}

\begin{proof}
    Let $C$ be a $(-K_X)$-line, i.e. a smooth rational curve $C$ such that $(C.-K_X)=1$. Let $F$ be the exceptional divisor of the blow-up $\phi:X\rightarrow V$, and $E$ be the exceptional divisor of the contraction $\pi:X\rightarrow \bP^3$. Set $H:=\pi^{*}\mtc{O}_{\bP^3}(1)$ and $L:=\pi^{*}\mtc{O}_{V}(1)$ to be two big and nef divisors. Then one has that $-K_X\sim L+H$, and that $(C.L+H)=1$. Thus either $(C.H)=0$, in which case $C$ is contained in $E$, or $(C.L)=0$, in which case $C$ is contained in $F$. Moreover, since $H+L$ is ample, these two cases cannot occur simultaneously. Notice that the lines contained in $F$ is parametrized by $\ell$ (a smooth rational curve), while the lines contained in $E$ is parametrized by a curve which is either irrational or singular.
\end{proof}

\begin{lemma}\label{lem:same automorphism}
    Let $X$ be the blow-up of a $(2,2)$-complete intersection $V$ in $\bP^5$ along a line $\ell$ which is not contained in the singular locus of $X$. Then there is an isomorphism $\Aut(X)\simeq \Aut(V,\ell)$.
\end{lemma}

\begin{proof}
    By the universal property of blow-ups, one can lift an automorphism of $V$ preserving $\ell$ to an automorphism of $X$, and hence there is a natural injection $\Aut(V,\ell)\hookrightarrow \Aut(X)$. Conversely, if $g\in \Aut(X)$ is an automorphism of $X$, then $g$ must preserve the exceptional divisor $F$ of $X\rightarrow V$ by Lemma \ref{lem: fano scheme of lines}. The rest of the proof is identical to \cite[Lemma 5.8]{LZ24}.
\end{proof}

\begin{remark}
   \textup{A different proof is given by \cite[Lemma 5.2]{JL18}.}
\end{remark}

\begin{prop}\label{prop:K-moduli as global quotient}
    There is an isomorphism of Artin stacks $$\mts{M}^{K}_{\textup{\textnumero 2.19}}\ \simeq [U^K/\PGL(6)],$$ where $U^K$ is the $\PGL(6)$-equivariant open subset of $U$ parametrizing pairs $(V,\ell)$ such that $\Bl_{\ell} V$ is a K-semistable $\bQ$-Fano variety.
\end{prop}

\begin{proof}
     Since $U^K$ parametrizes K-semistable $\mb{Q}$-Fano varieties, then by universality of K-moduli stacks, there exists a morphism $$\varphi:\ [U^K/\PGL(6)]\ \longrightarrow\ \mts{M}^K_{\textup{№2.19}}.$$ To show $\varphi$ is an isomorphism, we will construct the its inverse morphism. By the construction of K-moduli stacks, we know that $$\mts{M}^K_{\textup{№2.19}}\ \simeq\ [T/\PGL(N+1)],$$ where $T\subseteq \Hilb_{\chi}(\mb{P}^{N})$ is a locally closed subscheme of the Hilbert scheme with respect to the Hilbert polynomial $\chi(m):=\chi(X,-mrK_X)$ for $r\in \bZ_{>0}$ sufficiently divisible. By Corollary \ref{cor:stack-smooth}, we know that $T$ is irreducible and smooth. Let $\pi:\mts{Z}\rightarrow T$ be the universal family. By Lemma \ref{lem: fano scheme of lines}, there is a unique divisor $\mts{F}\subseteq \mts{Z}$ relative Cartier over $T$ such that each fiber $\mts{F}_t$ is the exceptional divisor of the blow-up along a line. 
    
    Consider the $\pi$-big and $\pi$-nef line bundle $\mts{R}:=-\omega_{\mts{Z}/T}+\mts{F}$ on $\mts{Z}$, whose restriction on each fiber $\mts{Z}_t$ is isomorphic to the pull-back of the line bundle $\mtc{O}_{\mb{P}^5}(2)$. Taking the $\pi$-ample model with respect to $\mts{R}$, one obtains a family $\mts{V}\rightarrow T$ such that each fiber $\mts{V}_t$ is isomorphic a $(2,2)$-complete intersection in $\bP^5$. 
    
    Take a fppf cover $\sqcup_i T_i\rightarrow T$ of $T$ such that the pull-back family $p_i:\mts{V}_i\rightarrow T_i$ can be embedded into is a trivial $\mb{P}^5$-bundle over $T_i$. Let $\mtc{L}=\sqcup_i \mtc{L_i}$ be the line bundle on $\sqcup_i \mts{V}_i$ which is isomorphic to the pull-back of $\mtc{O}_{\mb{P}^5}(1)$ over each $T_i$. By Kawamata-Viehweg vanishing, we know that $(p_i)_*\mtc{L}_i$ is a rank $6$ vector bundle over $T_i$. Let $\mtc{P}_i/T_i$ be the $\PGL(6)$-torsor induced by projectivized basis of $(p_i)_*\mtc{L}_i$. As the cocycle condition of $\{(p_i)_*\mtc{L}_i/T_i\}_i$ is off by $\pm1$, then we deduce that $\{\mtc{P}_i/T_i\}_i$ is a fppf descent datum descending to a $\PGL(6)$-torsor $\mtc{P}/T$, which is $\PGL(N+1)$-equivariant. This induces a morphism $$\varphi^{-1}:\mts{M}^K_{\textup{№2.19}}\longrightarrow [U^K/\PGL(6)].$$ Since both stacks are smooth, then $\varphi^{-1}$ is indeed the inverse of $\varphi$.
\end{proof}

\begin{theorem}\label{main: isom of moduli}
  Let $V$ be a $(2,2)$-complete intersection in $\mb{P}^5$, $\ell\subseteq V$ be a line which is not contained in the singular locus of $V$, and $X:=\Bl_{\ell}V$ be the blow-up of $V$ along $\ell$. Then $X$ is K-(semi/poly)stable if and only if $(V,\ell)$ is a GIT (semi/poly)stable point in $W$ (see Eq. (\ref{eq:defn W})) with respect to the linearized ample $\bQ$-line bundle $\eta+t_0\cdot\xi$, where $t_0=\frac{5}{26}$. Moreover, there is a natural isomorphism $$\mts{M}^K_{\textup{№2.19}}\ \simeq\  \mts{M}^{\GIT}(t_0) = \left[W^{\sst}(t_0)/\PGL(6) \right]\ =  \ \mts{M}^{\GIT}(\epsilon),$$ which descends to an isomorphism between their good moduli spaces $$\ove{M}^K_{\textup{№2.19}}\ \simeq\  \ove{M}^{\GIT}(t_0) = W^{\sst}(t_0)\sslash_{t_0}\PGL(6)\ = \ \ove{M}^{\GIT}(\epsilon).$$
\end{theorem}

\begin{proof}
Let $(V,\ell)\in W$ be a $t_0$-GIT semistable pair. Pick a family of pairs $\{(V_t,\ell_t)\}_{t\in T}$ over a smooth pointed curve $(0 \in T)$ such that $(V_0,\ell_0)=(V,\ell)$, and $\Bl_{\ell_t} V_t$ is K-stable and smooth for any $t\in T\setminus\{0\}$. After possibly a base change, there exists a K-polystable filling $X'$ by the properness of K-moduli spaces, which is isomorphic to the blow-up of a quartic del Pezzo threefold $V'$ along a line $\ell'$. By the separatedness of GIT quotients, we know that $(V,\ell)$ admits a special degeneration to $g \cdot (V',\ell')$ for some $g \in \PGL(6)$. Thus $X=\Bl_{\ell}V$ is K-semistable by openness of K-semistability. Suppose in addition that $(V,\ell)$ is GIT polystable. Let $X_0':=\Bl_{\ell'_0}V_0'$ be the K-polystable limit of $X$, for which we know that $(V_0',\ell_0')$ is GIT polystable by Proposition \ref{prop:K implies GIT}, and hence is S-equivalent to $(V,\ell)$. Therefore one has $(V,\ell)= g \cdot (V_0',\ell_0')$ for some $g \in \PGL(6)$ and $X$ is K-polystable. 

From the equivalence of K-(semi/poly)stability and GIT (semi/poly)stability, we can identify $W^K$ with $W^{\sst}(t_0)$. Then the isomorphism of moduli stacks (resp. moduli spaces) follows from Propositions \ref{prop:K-moduli as global quotient} and \ref{prop:no wall before 1/8}. 
\end{proof}

\begin{corollary}\label{cor:K-stability of smooth 2.19}
    If $X$ is a smooth Fano threefold in the family №2.19, then $X$ is K-stable. In particular, Thaddeus' moduli spaces of bundle stable pairs over smooth projective curve of genus $2$ are all K-stable, and the $\ove{M}^K_{\textup{№2.19}}$ is a modular compactification of moduli space of Thaddeus' moduli of bundle stable pairs on curves of genus $2$.
\end{corollary}

\begin{proof}
    We write $X$ to be the blow-up of a smooth quartic del Pezzo threefold $V\subseteq \bP^5$ along a line $\ell$. Since $V$ is smooth, then $V$ is GIT stable, and hence the pair $(V,\ell)$ is $\epsilon$-GIT stable for $0<\epsilon\ll1$. It then follows from Theorem \ref{main: isom of moduli} that $X$ is K-stable.
\end{proof}

Finally, by studying the deformation, we will show that the moduli stack $\mts{M}^K_{\textup{№2.19}}$ is smooth.

\begin{prop}\label{prop:deformation}
    Let $X$ be the blow-up of a $(2,2)$-complete intersection $V$ in $\bP^5$ along a line $\ell$ which is not contained in the singular locus of $X$. Assume in addition that $V$ is normal and that $X$ has Gorenstein canonical singularities. Then $\Ext^2(\Omega^1_{X},\mtc{O}_X)=0$. In particular, there are no obstructions to deformation of $X$.
\end{prop}

\begin{proof}
   The blow-up morphism $\phi:X\rightarrow V\subseteq \mb{P}^5$ and it sits into a commutative diagram $$\xymatrix{
 & X \ar[rr]^{\phi} \ar@{^(->}[d]  &  &  V \ar@{^(->}[d]\\
 & \wt{\mb{P}}:=\Bl_{\ell}\mb{P}^5  \ar[rr]^{\quad \psi}  &   & \mb{P}^5. \\
 }$$ By taking $R\Hom(\cdot, \mtc{O}_X)$ of the short exact sequence $$0 \ \longrightarrow \ (N_{X/\wt{\mb{P}}})^*\ \longrightarrow\ \Omega^1_{\wt{\mb{P}}}|_X\  \longrightarrow\   \Omega^1_X \ \longrightarrow\   0,$$ it suffices to show that $H^1(X,N_{X/\wt{\mb{P}}})=0$ and $\Ext^2(\Omega^1_{\wt{\mb{P}}}|_X,\mtc{O}_X)=0$. 

 Let $L:=\psi^{*}\mtc{O}_{\mb{P}^5}(1)$, $F\simeq \bP^1\times \mb{P}^3$ be the $\psi$-exceptional divisor, and $Q:=F\cap X$ be the $\phi$-exceptional locus. Notice that $X$ is a complete intersection in $\wt{\bP}$, and one has $$N_{X/\wt{\mb{P}}}\ \simeq \ \mtc{O}_X(2L-F)\oplus\mtc{O}_X(2L-F),$$ which is ample, and hence $H^1(X,N_{X/\wt{\mb{P}}})=0$ by Kawamata-Viehweg vanishing.
 
 To show the second vanishing, let us consider the short exact sequence $$0 \longrightarrow \psi^*\Omega^1_{\mb{P}^5}|_X \ \longrightarrow\ \Omega^1_{\wt{\mb{P}}}|_X \ \longrightarrow\  \Omega^1_{F/\ell}|_X \ \longrightarrow \ 0,$$ where we use the fact that the $F$ intersects generically transversely with $X$. It suffices to show that $\Ext^2(\psi^*\Omega^1_{\mb{P}^5}|_X,\mtc{O}_X)=0$ and $\Ext^2(\Omega^1_{F/\ell}|_X,\mtc{O}_X)=0$. Consider the pull-back of the Euler sequence  $$0 \ \longrightarrow\  \mtc{O}_X \ \longrightarrow\ (L|_X)^{\oplus 6 } \ \longrightarrow \  \psi^{*}T_{\mb{P}^5}|_X \ \longrightarrow \  0.$$ One has that $H^3(X,\mtc{O}_X)=0$ and $H^2(X,L|_X)=0$ by Kawamata-Viehweg vanishing, and hence $$\Ext^2(\psi^*\Omega^1_{\mb{P}^5}|_X,\mtc{O}_X)\ \simeq\  H^2(X,\psi^{*}T_{\mb{P}^5}|_X)\ =\ 0.$$ Let $p:F\rightarrow \bP^3$ be the natural projection. Then one has $\Omega^1_{F/\ell}\simeq  p^{*}\Omega^1_{\bP^3}$. Consider the pull-back of the Euler sequence on $\bP^3$ (each entry viewed as a torsion sheaf on $\wt{\mb{P}}$) $$0 \ \longrightarrow \ p^{*}\Omega^1_{\bP^3}\ \longrightarrow \ p^{*}\mtc{O}_{\bP^3}(-1)^{\oplus 4 } \ \longrightarrow \  p^{*} \mtc{O}_{\bP^3} \longrightarrow  0.$$ Twisting by $\mtc{O}_X$, one obtains $$0 \ \longrightarrow \ p^{*}\Omega^1_{\bP^3}|_X\ \longrightarrow \ p^{*}\mtc{O}_{\bP^3}(-1)^{\oplus 4 }|_X \ \longrightarrow \   \mtc{O}_F|_X \ \longrightarrow \  0.$$ We now need to prove that $\Ext^2(\mtc{O}_F(-1)|_X,\mtc{O}_X)=0$ and $\Ext^3(\mtc{O}_F|_X,\mtc{O}_X)=0$. Since $\mtc{O}_F|_X=\mtc{O}_Q$, then by Serre duality one has $$\Ext^3(\mtc{O}_F|_X,\mtc{O}_X)\ \simeq\  H^0(X,\omega_X\otimes \mtc{O}_F|_X)^{*} \ \simeq\  H^0(Q,\omega_X|_Q)^{*}\ =\ 0$$ as $\omega_X$ is anti-ample. Similarly, we have that $$\Ext^2(p^{*}\mtc{O}_{\bP^3}(-1)|_X,\mtc{O}_X)\ \simeq\  H^1(X,\omega_X\otimes \mtc{O}_F(-1)|_X)^{*} \ \simeq\  H^1(X,2(F-L)|_Q)^{*},$$ where we use that $\mtc{O}_F(-1)=\mtc{O}_{F}(F)$ and $\omega_X\sim F-2L$. Consider the short exact sequence $$0\ \longrightarrow \mtc{O}_X(F-2L)\ \longrightarrow \mtc{O}_X(2F-2L)\ \longrightarrow \ \mtc{O}_Q(2F-2L)\ \longrightarrow \ 0.$$ Since $2L-F$ (resp. $L-F$) is ample (resp. big and nef), then by Kawamata-Viehweg Vanishing theorem $\mtc{O}_X(F-2L)$ (resp. $\mtc{O}_X(2F-2L)$) has no middle cohomology, and hence $$H^1(X,2(F-L)|_{Q})\ =\ H^1(Q,2(F-L)|_{Q})\ =\ 0$$ as desired.
\end{proof}

\begin{corollary}\label{cor:stack-smooth}
    The K-moduli stack $\mts{M}^K_{\textup{№2.19}}$ is smooth, and its good moduli space $\ove{M}^K_{\textup{№2.19}}$ is normal. Moreover,  $\mts{M}^K_{\textup{№2.19}}$ is a smooth connected component of $\mts{M}^K_{3, 26}$.
\end{corollary}

\section{Further discussion}

In this section, we will propose several problems related to this circle of ideas, which merit further study and hold significant potential for exploration.

\subsection{More on K-moduli of \textup{№2.19}}\label{subsec: more on K-moduli}

It would be interesting to classify all K-(semi/poly)stable limits in the family \textup{№2.19}. To achieve this, one can classify all VGIT quotients $\ove{M}^{\GIT}(t)$ for $t \in (0, \frac{1}{2})$, which is feasible. 

Another valuable direction is to describe the fibers of the forgetful morphism $$\phi: \ove{M}^K_{\textup{№2.19}} \ \longrightarrow \  \ove{M}^K_{\textup{№1.14}}$$ as compactified Jacobians of a genus 2 curve $C$ (or its quotient by $\Aut(C)$). The K-moduli space $\ove{M}^K_{\textup{№1.14}}$ is isomorphic to the GIT moduli space of $(2,2)$-complete intersections in $\bP^5$ (ref. \cite[Theorem 4.9]{SS17}), which is further isomorphic to the moduli space of pseudo-stable curves of genus 2, i.e. the final log canonical model of $\ove{M}_2$ in the Hassett-Keel program. For a smooth genus 2 curve corresponding to the $(2,2)$-complete intersection $V$, the fiber of $\phi$ is isomorphic to the quotient $F_1(V)\sslash\Aut(V)$, which is a (singular) Kummer K3 surface isomorphic to $\Jac(C)\sslash\Aut(C)$ (cf. \cite{Rei72}).

This problem becomes particularly intriguing when $C$ is a singular pseudo-stable curve. For instance, if $C$ is irreducible with one node and normalization $\wt{C}$, the K-polystable $(2,2)$-complete intersection $V_0$ whose discriminant recovers $C$ as a double cover has two $A_1$-singularities and admits a $\bG_m$-action. There is a unique K-semistable $(2,2)$-complete intersection $V$, up to isomorphism, that degenerates isotrivially to $V_0$ and it has one $A_1$-singularity. The sublocus of $F_1(V)$ parametrizing lines containing the singularity is isomorphic to $\Jac(\wt{C}) \simeq  \wt{C}$, while the sublocus $F_1^\circ(V_0)$ of $F_1(V_0)$ parametrizing lines avoiding singularities is isomorphic to $\wt{C} \times \bG_m$, collapsing to $\wt{C}$ under the $\bG_m$-action on $V_0$. By Proposition \ref{prop:K implies GIT}, if $X = \Bl_{\ell}V$ is K-semistable, then the line $\ell$ lies in the smooth locus of $V$. Hence, the fiber $\phi^{-1}([V_0])$ should admit a stratification $$\big(F_1^{\circ}(V_0)\sqcup \wt{C} \big)\sslash \Aut(C)\ \simeq \ F_1(V)\sslash \Aut(C),$$ where $F_1(V)$ is conjectured to be isomorphic to the compactified Jacobian of $C$. Similarly, we expect analogous behavior for nodal curves with two or three singularities, corresponding to K-polystable $(2,2)$-complete intersections with four or six $A_1$-singularities.

However, the role of the bicuspidal rational curve in this framework remains unclear. The moduli space of vector bundles over cuspidal curves deserves further investigation.

\subsection{Related problems}

In this section, we present more related open problems. Addressing some of these may require developing more moduli theory and new tools for studying higher-dimensional singularities.

\subsubsection{Moduli of bundle (and bundle stable pairs) on higher genus curves}

The moduli spaces of rank two vector bundles and Thaddeus' moduli spaces of stable pairs over higher genus curves remain poorly understood. Let $C$ be a general smooth projective curve of genus $g\geq 3$, and let $\Lambda$ be a line bundle of degree $2g-1$ on $C$. Recent work in \cite{Kel24} shows that the moduli space $\ove{N}_C(2,\Lambda)$ of stable vector bundles of rank two and determinant $\Lambda$ is K-stable.

\begin{Prbm}\label{prob:1}
 Can one describe the irreducible component of the K-moduli space, a general point of which parametrizes $\ove{N}_C(2,\Lambda)$?
\end{Prbm}

\begin{Prbm}\label{prob:2}
    For a smooth projective curve of genus $g\geq3$ and a line bundle $\Lambda$ of degree $2g-1$ on $C$, is the moduli space $\ove{M}_{C,g-1}(\Lambda)$ of bundle stable pairs K-(semi)stable?
\end{Prbm}

\begin{Prbm}\label{prob:3}
    If $\ove{M}_{C,g-1}(\Lambda)$ is K-semistable for some $C$ of genus $g\geq 3$ and $\Lambda\in\Pic^{2g-1}(C)$, then can we explicitly describe the K-moduli compactification of $\ove{M}_{C,g-1}(\Lambda)$? In particular, is there a morphism from the K-moduli stack (space) of $\ove{M}_{C,g-1}(\Lambda)$ to the K-moduli stack (space) of $\ove{N}_{C}(2,\Lambda)$?
\end{Prbm}

\subsubsection{Pencil of quadrics, hyperelliptic curve and Fano varieties}

Another direction for generalizing the results of this paper is to study $(2,2)$-complete intersections of quadrics in higher-dimensional projective spaces. The work in \cite{Rei72} establishes a correspondence between $(2,2)$-complete intersections $V$ in $\bP^{2n+3}$ and hyperelliptic curves $C$ of genus $n+1$. Furthermore, the Fano scheme of $n$-planes on $V$ is isomorphic to $\Jac(C)$.

Let $V$ be a smooth $(2,2)$-complete intersection in $\bP^{2n+3}$, and let $\Lambda$ be an $n$-plane on $V$. The blow-up $X:= \Bl_{\Lambda}V$ is then a smooth Fano variety of dimension $2n+1$. By \cite[Theorem 1.1]{SS17}, the K-moduli space of $(2,2)$-complete intersections is isomorphic to the GIT quotient $$\Gr\big(2,H^0(\bP^{2n+4},\mtc{O}_{\bP^{2n+4}}(2))\big)\sslash \PGL(2n+4).$$ Thus, a natural direction for generalization is to explore the following.

\begin{Prbm}\label{prob:4}
    Keep the notations as above. 
    \begin{enumerate}
        \item Is a general blow-up $X=\Bl_{\Lambda}V$ K-stable?
        \item Is any K-semistable limit (if exists) also isomorphic to the blow-up of a $(2,2)$-complete intersection along an $n$-plane?
        \item Can one identify the K-moduli space of $X$ (if non-empty) with some VGIT quotient of pairs $(V,\Lambda)$?
        \item Is there a forgetful morphism from the K-moduli stack of $X$ to the K-moduli stack of $V$?
    \end{enumerate}
\end{Prbm}

Problem \ref{prob:1} might be achieved using the admissible flag method developed in \cite{AZ22} via induction, but Problem \ref{prob:2} is especially hard, and hence Problem \ref{prob:3} and Problem \ref{prob:4} seem unreachable with our current methods.

\subsubsection{K-moduli spaces of Fano threefolds}

The moduli continuity approach, combined with the use of moduli spaces of K3 surfaces, is expected to be applicable to other families of Fano threefolds with Picard rank two or higher. In particular, the following two families resemble those that have already been solved in this paper as well as \cite{LZ24}.

\begin{Prbm}
    Describe the K-moduli stacks (spaces) of the following two families, and identify them with some natural GIT moduli spaces:
    \begin{enumerate}
        \item №2.11: blow-up of a smooth cubic hypersurface in $\bP^4$ along a line;
        \item №2.16: (Fano's last Fano) blow-up of a smooth $(2,2)$-complete intersection in $\bP^5$ along a conic.
    \end{enumerate}
\end{Prbm}

\appendix
\section[tocentry={Families of smooth K-stable Fano over $\bP^1$}]{Families of smooth K-stable Fano over $\bP^1$}\label{appendix}

\begin{center}
\chapterauthor{Benjamin Church and Junyan Zhao}
\end{center}

We discuss a question of Javanpeykar and Debarre in the context of Fano family №2.19. 

\newcommand{\PP}{\mathbb{P}}
\newcommand{\ZZ}{\mathbb{Z}}

\begin{question}[{\cite[Question 4.5]{Deb23}}]
Does there exist a non-isotrivial family of smooth Fano threefolds over $\PP^1$?    
\end{question}

We would like to propose a related question:

\begin{question}
Are there non-isotrivial families of K-stable smooth Fano varieties over $\PP^1$? 
\end{question}

To the authors knowledge, this is still an open question. The known examples of non-isotrivial smooth Fano fibrations over $\bP^1$ consisting of Gushel-Mukai fourfolds \cite{DK20} (cf.\ \cite[Section 4.4]{Deb23}) or spinor sixfolds \cite{Kuz18} are -- to the authors knowledge -- not known to have every fiber K-stable.
\par 
Here we adress these questions for Fano family №2.19. Let $\mts{M}^K_{\textup{№2.19}}$ be the K-moduli stack of the family №2.19, $\ove{M}^K_{\textup{№2.19}}$ be its good moduli space, and $\ove{M}^{K,\textup{smooth}}_{\textup{№2.19}}$ be the open subset of $\ove{M}^K_{\textup{№2.19}}$ parametrizing smooth (K-stable) members. Then the followings hold:

\begin{enumerate}
    \item $\ove{M}^{K,\textup{smooth}}_{\textup{№2.19}}$ admits a map from an elliptic curve without automorphisms that lifts to the stack;
    \item $\ove{M}^{K,\textup{smooth}}_{\textup{№2.19}}$ contains a dense set of rational curves (that lie in infinitely many numerical classes); however
    \item no such rational curve lifts to the stack $\mts{M}^K_{\textup{№2.19}}$, and indeed there is no map $\PP^1 \to \mts{M}^K_{\textup{№2.19}}$ whose image in $\ove{M}^K_{\textup{№2.19}}$ is nonconstant and contained in $\ove{M}^{K,\textup{smooth}}_{\textup{№2.19}}$.
\end{enumerate}

In order to prove these results, we need to study automorphisms of smooth Fano varieties in family №2.19.

\begin{prop}\label{prop:relation aut}
    Let $V$ be a smooth quartic del Pezzo threefold, and $f:\mts{Q}\rightarrow \Phi$ be a pencil of quadric fourfolds in $\bP^5$ with $V$ as base locus. Let $C$ be the smooth genus $2$ curve obtained by taking double cover of $\Phi$ branched along the discriminant locus of $f$ and $A$ be the Jacobian of $C$. Then $\Aut(V)$ induces an action on $C$, fitting in the following diagram of short exact sequences:
    \begin{center}
    \begin{tikzcd}
    & 0 \ar[r] & \{\pm1\} \ar[r] & \Aut(C) \ar[r] & \Aut(\Phi;p_0 + \cdots + p_5) \ar[r] & 0 \\
    &  0 \ar[r]  & (\bZ/2\bZ)^5 \ar[r] \ar[u] & \Aut(V) \ar[r] \ar[u] & \Aut(\Phi;p_0 + \cdots + p_5) \ar[u, equals] \ar[r] & 0\\
    & & A[2] \ar[r, equals] \ar[u] & A[2] \ar[u] 
    \end{tikzcd}
    \end{center}
\end{prop}

\begin{proof}
    First note that each $g\in \Aut(V)$ is induced by an automorphism of the ambient $\bP^5$, and hence induces an action on the pencil $f:\mts{Q}\rightarrow \Phi$ of quadrics defining $V$. The genus two curve $C$ parametrizes the families of $2$-planes contained in the fiber of $f$. As a consequence, $g$ induces an automorphism on $C$ such that $C\rightarrow \bP^1$ is equivariant.

    Fix a base point $\ell_0$ of $A:=F_1(V)$ (as the origin of the abelian surface) so that the subgroup of $\Aut(V)$ generated by $\sigma_i:x_i\mapsto -x_i$ acts as $A[2]\times \{\pm1\}$, and the induced $(-1)$-action on $C$ is the hyperelliptic involution, while the induced $A[2]$-action on $C$ is the trivial action. Moreover by \cite[Lemma 3.2]{Bho10}, the kernel of $\Aut(V)\rightarrow \Aut(\Phi;p_0+\cdots+p_5)$ is isomorphic to $(\bZ/2\bZ)^2$ generated by $\sigma_i$ which swap the sign of $x_i$.

    Now it remains to show that every automorphism $g$ of $(\Phi;p_0+\cdots+p_5)$ lifts to an automorphism of $\mts{Q}$, which induces an automorphism of $V$. We may assume that $p_0,...,p_5\neq\infty$, and the pencil $\mts{Q}$ is defined by $\sum_{i=0}^5(u-vp_i)x_i^2=0$, where $[u:v]$ is the homogeneous coordinate on $\Phi\simeq \bP^1$. Let $\tau$ be the element in $\mtf{S}_6$ which permutes $\{0,...,5\}$ induced by $g\in \Aut(\Phi;p_0+\cdots+p_5)$. Then $x_i\mapsto x_{\tau(i)}$ is an automorphism of $\Aut(\bP^5)$ which preserves $V$ and is sent to $g$ under $\Aut(V)\rightarrow \Aut(\Phi;p_0+\cdots+p_5))$.
\end{proof}

\begin{example}
  The automorphism group of a smooth genus $2$ curve is known to be one of the following (cf. \cite{Bol1888,Igu60}):
    \begin{enumerate}
        \item $G$ is cyclic of order $2,10$;
        \item $G$ is isomorphic to $D_4,D_8$ or $D_{12}$;
        \item $G$ is isomorphic to $\wt{\mtf{S}}_4\simeq \GL(2,3)$;
        \item $G$ is isomorphic to $2D_{12}$.
    \end{enumerate}
Let $V$ be the smooth $(2,2)$-complete intersection defined by $Q=\bV(x_0^2+\cdots+x_5^2)$ and $Q'=\bV(x_0^2+\zeta x_1^2+\cdots+\zeta^4 x_4^2)$, where $\zeta$ is a primitive 5th root of unity. This pencil of quadrics corresponds to the genus two curve $C$ with a hyperelliptic double cover to $$\left(\bP^1,\ \frac{1}{2}\big([0]+[1]+[\zeta]+\cdots+[\zeta^4]\big)\right).$$ The automorphism group of $\bP^1$ preserving the discriminant locus is $\bZ/5\bZ$ generated by $\zeta:z\mapsto \zeta\cdot z$, and the automorphism group of $C$ is $\bZ/10\bZ$ generated by $\zeta\cdot \tau$, where $\tau$ is the hyperelliptic involution. The automorphism $\zeta$ lifts an element in $\Aut(V)$ given by $$x_0\mapsto\zeta^3\cdot x_1 \ \ ,\ \ \cdots \ \ ,\ \  x_4\mapsto \zeta^3\cdot x_0\ \ ,\ \ x_5\mapsto x_5.$$

\end{example}

\begin{lemma} \label{lemma:automorphisms}
    Let $V$ be a smooth $(2,2)$-complete intersection in $\bP^5$. Then $\Aut(V)$ is finite and acts faithfully on $F_1(V)$. Moreover, if $V$ is general, then the followings hold:
    \begin{enumerate}
        \item $\Aut(V)$ is isomorphic to $(\bZ/2\bZ)^5$, which acts trivially on the corresponding pencil of quadrics;
        \item there is a choice of isomorphism $A \cong F_1(V)$ with $A$ an abelian surface such that $\Aut(V)$ acts on $A$ via the standard action of $A[2]\times \{\pm 1\}$, where $A[2]$ is the subgroup of $2$-torsion points on $A$, and $-1$ is the involution of $A$ by sending a point to its negative (up to the choice of base point);
        \item the quotient of $F_1(V)$ by $\Aut(V)$ is isomorphic to the (singular) Kummer K3 surface with 16 $A_1$-singularities associated to $A$.
    \end{enumerate} 
\end{lemma}

\begin{proof}
The finiteness of $\Aut(V)$ and the faithful action part are known by \cite[Corollary 4.2.3]{KPS18}. The items (1) and (2) follow from \cite[Theorem 1.3(2)]{CPZ24} and \cite[Remark 4.2.4]{KPS18}. For (3), this follows immediately from the description (2) where the projection $F_1(V) \to F_1(V) / \Aut(V)$ factors as the composition $A \xrightarrow{\times 2} A \to A / \{ \pm 1 \} = K$ via first quotienting by $A[2]$ and then by $\pm 1$, and $K$ is the singular model of the Kummer K3 surface.  
\end{proof}

\begin{corollary}
    Let $V$ be a general $(2,2)$-complete intersection in $\bP^5$, $\ell \subset V$ a line, and $X := \Bl_{\ell}V$ the blow-up along the line $\ell$. Then $\Aut(X)$ is
    \begin{enumerate}
        \item isomorphic to $\bZ/2\bZ$ if $\ell$ maps to a $4$-torsion point under the isomorphism $F_1(V)\stackrel{\sim}{\rightarrow}A$ in Lemma~\ref{lemma:automorphisms} ;
        \item trivial otherwise.
    \end{enumerate}
\end{corollary}

\begin{proof}
By Lemma \ref{lem:same automorphism}, the automorphisms of $X$ are those of $V$ preserving $\ell$. Since the automorphism group of $V$ acts faithfully on $F_1(V)$, we see $\Aut(X)$ is identified with the subgroup of $A[2] \times \{ \pm 1 \}$ acting on $A$ that fixes a point $P_{\ell} \in A$ corresponding to $\ell$. If $P_{\ell} \in A[4]$ then it is fixed by $-1$ composed with translation by $2 P_{\ell}$ but not by any other element. Indeed translations have no fixed points so if $f(x) = -x + T$ for $T \in A[2]$ fixes $P_{\ell}$, then we find $T = 2 P_{\ell}$. Hence, notice that the only fixed point of $x \mapsto -x + T$ for $T \in A[2]$ is $2 x = T$ i.e.\ $x \in A[4]$.
\end{proof}

\begin{lemma} \label{lemma:fiber_over_BAut}
Let $V$ be a smooth $(2,2)$-complete intersection in $\bP^5$, and let $\mts{M}^K_V$ be the fiber of $\mts{M}^K_{\textup{№2.19}}\rightarrow \mts{M}^K_{\textup{№1.14}}$ over the residual gerbe $B \Aut(V) \to \mts{M}^K_{\textup{№1.14}}$ at $[V]$. Then there is a canonical isomorphism 
\[ \mts{M}^K_V \cong [F_1(V) / \Aut(V)] \]
\end{lemma}

\begin{proof}
This is a standard fact: given that the fiber over the field-valued point $\Spec{\bC} \to \mts{M}^K_{\textup{№1.14}}$ corresponding to $V$ is $F_1(V)$. Then the fiber over $B \Aut(V)$ is $[F_1(V) / \Aut(V)]$.
\end{proof}

\begin{theorem} The following statements hold:
\begin{enumerate}
    \item There is a non-isotrivial family of smooth Fano threefolds №2.19 parametrized by a stacky $\bP^1$, i.e.\ a proper DM stack $\mtc{P}^1$ with finitely many stacky points, whose coarse space is isomorphic to $\bP^1$,
    \item there is a non-constant morphism from $\bP^1 \to \ove{M}^{K}_{\textup{№2.19}}$ whose image is contained in the locus parametrizing smooth K-stable Fano threefolds of family №2.19. In fact, such curves are dense and occupy infinitely many numerical classes. However, none lift to $\mts{M}^K_{\textup{№2.19}}$ and moreover,
    \item there are no non-trivial families of smooth Fano threefolds in №2.19 over $\PP^1$.
\end{enumerate}
\end{theorem}

\begin{proof}
By Theorem \ref{thm:main2} (1), the moduli stack $\mts{M}^{K, \text{smooth}}_{\textup{№2.19}}$ consists entirely of $K$-stable elements and is in particular a DM stack. Hence any nontrivial (meaning not factoring through a point) morphism $\bP^1 \to \mts{M}^{K, \text{smooth}}_{\textup{№2.19}}$ must produce a nontrivial map to the coarse space $\ove{M}^{K}_{\textup{№2.19}}$. To see this, suppose  $\bP^1 \to \ove{M}^{K}_{\textup{№2.19}}$ had image $x$, then the map from $\bP^1$ to the residual gerbe $B G_x$ of $\mts{M}^{K, \text{smooth}}_{\textup{№2.19}}$ over $x$ would produce a $G_x$-torsor on $\bP^1$ which must be trivial since $\bP^1$ is simply-connected. The torsor is trivial so $\bP^1 \to \mts{M}^{K, \text{smooth}}_{\textup{№2.19}}$ factors through $\Spec{k(x)} \to \mts{M}^{K, \text{smooth}}_{\textup{№2.19}}$.
\par 
To conclude, we must show that there is a dense set of maps $\bP^1 \to \ove{M}^{K}_{\textup{№2.19}}$ and that no such map can lift to the stack.
\par 
We first show that there is no $\bP^1$ mapping nontrivially to the stack $\mts{M}^K_V$,
using the notation of Lemma~\ref{lemma:fiber_over_BAut}. There is an \'{e}tale cover $F_1(V) \to \mts{M}^K_V$ thus any map from $\bP^1 \to \mts{M}^K_V$ is trivial. Otherwise it would have an \'{e}tale cover whose image maps nontrivially to an abelian variety. Furthermore, this shows that $\mts{M}^{K, \text{smooth}}_{\textup{№2.19}}$ contains no non-trivial map from $\PP^1$. As indicated above, every smooth member is $K$-stable, the automorphism groups are finite, and hence any isotrivial family over $\PP^1$ is actually trivial. Since there is no non-constant $\bP^1 \to  \mts{M}^{K, \text{smooth}}_{\textup{№1.14}}$ by Hodge theory (or by using the identification with the moduli space of curves $\mts{M}_2$ at the level of coarse spaces), any such map $\PP^1 \to \mts{M}^{K, \text{smooth}}_{\textup{№2.19}}$ must lie in some fiber of $\mts{M}^{K, \text{smooth}}_{\textup{№2.19}}\rightarrow \mts{M}^{K, \text{smooth}}_{\textup{№1.14}}$ so we are done.
\par 
Now choose $V$ to be a $(2,2)$ complete intersection general enough that Lemma~\ref{lemma:automorphisms} applies. By \cite{CGL22}, the singular Kummer surface of Lemma~\ref{lemma:automorphisms} (3) 
$$\Sigma_V\ := \ F_1(V) / \Aut(V)$$ contains infinitely many rational curves (not necessarily smooth). Notice that all these rational curves pass through the singular points (otherwise they would lift to the \'{e}tale cover $F_1(V)$ which contains no rational curves). Varying over $V$ general enough to have automorphism group as described in Lemma~\ref{lemma:automorphisms} we get a dense collection of rational curves on $\ove{M}^{K}_{\textup{№2.19}}$. Pulling these rational curves back along 
$$ \mts{M}^K_V = [F_1(V) / \Aut(V)] \to F_1(V) / \Aut(V) $$
gives a collection of stacky curves $\mtc{P}^1$ with coarse space is $\bP^1$ whose images in $\mts{M}^K_V \to \mts{M}^{K, \text{smooth}}_{\textup{№2.19}}$ are dense.
\end{proof}

Finally, we make some remarks on the relations of different moduli spaces.
Let 
$$\mts{M}_{\mts{C}_2 / \mts{M}_2}(2,3)\ \longrightarrow \ \mtc{Pic}^3_{\mts{C}_2/\mts{M}_2} \ \longrightarrow\ \mts{M}_2$$ 
denote the moduli stack of rank two stable vector bundles of degree $3$ mapping to the universal Picard stack of degree $3$ line bundles, i.e.\ the fiber over the point $(C,\Lambda)\in \mtc{Pic}^3_{\mts{C}_2/\mts{M}_2}(\bC)$ is isomorphic to $\mts{M}_C(2,\Lambda)$. Let $\mts{Z}\subseteq \mts{M}_{\mts{C}_2}(2,3)$ be the Brill-Noether locus which parametrizes vector bundles $E$ with $h^0(C,E)\geq 2$, and let $\mts{X}\rightarrow \mts{M}_{\mts{C}_2}(2,3)$ be the blow-up of $\mts{M}(\mts{C}_2)$ along $\mts{Z}$. Then $\mts{X}\rightarrow \mtc{Pic}^3_{\mts{C}_2/\mts{M}_2}$ is a family of smooth K-stable Fano №2.19, and there is a natural morphism 
$$\phi\ :\ \mtc{Pic}^3_{\mts{C}_2/\mts{M}_2}\ \longrightarrow \ \mts{M}^{K, \text{smooth}}_{\textup{№2.19}} $$ by the universality of K-moduli stack. By \cite[Lemma 5.7]{JL18}, there is an isomorphism between the stack of lines on smooth $(2,2)$ complete intersections in $\bP^5$ and $\mts{M}^{K, \text{smooth}}_{\textup{№2.19}}$ given by blowing up the line. We guess the morphism
\[ \phi\ :\ \mtc{Pic}^3_{\mts{C}_2/\mts{M}_2}\ \longrightarrow \ \mts{M}^{K, \text{smooth}}_{\textup{№2.19}} \]
has the structure of a $\mathbb{G}_m$-gerbe. Also, notice that the natural map $\mtc{Pic}^3_{\mts{C}_2/\mts{M}_2}\ \longrightarrow \ \mts{M}_2$ factors as 
\[ \mtc{Pic}^3_{\mts{C}_2/\mts{M}_2}\ \longrightarrow \ \mts{M}^{K, \text{smooth}}_{\textup{№2.19}} \ \longrightarrow \ \mts{M}^{K, \text{smooth}}_{\textup{№1.14}} \ \longrightarrow \  \mts{M}_2 ,\]
where the last morphism between the two smooth DM stacks is conjecturally a $(\bZ/2\bZ)^4$-gerbe.

\textsc{Department of Mathematics, Stanford University, California 94305} \\
\indent \textit{E-mail address:} \href{mailto:bvchurch@stanford.edu}{bvchurch@stanford.edu}

\textsc{Department of Mathematics, University of Maryland, College Park, Maryland 20742} \\
\indent \textit{E-mail address:} \href{mailto:jzhao81@uic.edu}{jzhao81@umd.edu}

Conflicts of interest: none.

\bibliographystyle{alpha}
\bibliography{citation}

\end{document}